\documentclass{svjour2}                    % onecolumn
\smartqed  % flush right qed marks, e.g. at end of proof
\usepackage{graphicx}
\newcommand{\eps}{\epsilon}

\newcommand{\U}{\mathbb{U}}

\newcommand{\HH}{\mathbb{H}}

\newcommand{\LF}{\mathcal{F}}

\newcommand{\LE}{\mathcal{E}}
\newcommand{\LG}{\mathcal{G}}

\newcommand{\LH}{\mathcal{H}}
\newcommand{\R}{\mathbb{R}}
\newcommand{\C}{\mathbb{C}}
\newcommand{\N}{\mathbb{N}}
\newcommand{\Z}{\mathbb{Z}}
\newcommand{\E}{\mathbb{E}}
\newcommand{\PP}{\mathbb{P}}

\newcommand{\one}{1}

\newcommand{\cond}{\,|\,}

\newcommand{\dist}{dist}
\newcommand{\SLE}{\text{SLE}}

\newcommand{\im}{Im}
\newcommand{\re}{Re}

\usepackage{amsmath}
\usepackage{amssymb}
\usepackage{comment}

\begin{document}

\title{Polychromatic Arm Exponents for the Critical Planar FK-Ising Model}
\author{Hao Wu}
\institute{Yau Mathematical Sciences Center of Tsinghua University\at JingZhai 311, Tsinghua University, 100084, Beijing, China\\\email{hao.wu.proba@gmail.com}}
\date{Received: 03 July 2017 / Accepted: date}

% your contribution title if the original one is too long%

\maketitle

\begin{abstract}
Schramm Loewner Evolution (SLE) is a one-parameter family of random planar curves introduced by Oded Schramm in 1999 as the candidates for the scaling limits of the interfaces in the planar critical lattice models. 
This is the only possible process with conformal invariance and a certain ``domain Markov property". 
In 2010, Chelkak and Smirnov proved the conformal invariance of the scaling limits of the critial planar FK-Ising model which gave the convergence of the interface to $\SLE_{16/3}$.  
We derive the arm exponents of $\SLE_{\kappa}$ for $\kappa\in (4,8)$. Combining with the convergence of the interface, we derive the arm exponents of the critical FK-Ising model. We obtain six different patterns of boundary arm exponents and three different patterns of interior arm exponents of the critical planar FK-Ising model on the square lattice. 
\end{abstract}

%%%%%
\section{Introduction}
%%%%%
Fortuin and Kasteleyn introduced the random-cluster model in 1969. 
The random-cluster model is a probability measure on edge configurations of a finite graph $G$ where each edge is open or closed, and the probability of a configuration is proportional to 
\[p^{\#\text{open edges}}(1-p)^{\#\text{closed edges}}q^{\#\text{clusters}},\]
where $p\in [0,1]$ is the edge weight and $q>0$ is the cluster weight. 
The graph $G$ will always be a finite subgraph of $\Z^2$ in this paper. 
When $q\ge 1$, the model enjoys FKG inequality which makes it possible to consider infinite volume measures: infinite volume measures may be constructed on $\Z^2$ or $\Z\times\N$ by taking limits of measures on finite increasing subgraphs of $\Z^2$ or $\Z\times\N$ respectively. 
When $q<1$, little is know for the model.
The random cluster model is related to various models: percolation, Ising model etc. and the readers could consult \cite{CDParafermionic} for background. 
For $q\ge 1$, there exists a critical value $p_c$ for each $q$ such that for $p>p_c$, any infinite volume measure has an infinite cluster almost surely; whereas for $p<p_c$, any infinite volume measure has no infinite cluster almost surely. This dichotomy does not tell what happens at criticality $p=p_c$ and the critical phase is of great interest. It turns out to have continuity of the phase transition for $1\le q<4$, proved in \cite{DuminilSidoraviciusTassionContinuityPhaseTransition}. 
When $q\in (0,4]$, the critical phase is believed to be conformally invariant and the interface at criticality is conjectured to converge to $\SLE_{\kappa}$ where 
\begin{equation}\label{eqn::kappa_q}
\kappa=4\pi/\arccos(-\sqrt{q}/2).
\end{equation}

This conjecture is proved for $q=1$ for Bernoulli site percolation on triangular lattice \cite{SmirnovPercolationConformalInvariance} and
is proved for $q=2$ on isoradial graphs by the celebrated works of Chelkak and Smirnov \cite{ChelkakSmirnovIsing}, \cite{CDCHKSConvergenceIsingSLE}. When $q=2$, the random-cluster model is also called the FK-Ising model. In this paper, we derive the polychromatic arm exponents of FK-Ising model on the square lattice $\Z^2$.

In the random-cluster model, an arm is a primal-open path (type 1) or a dual-open path (type 0). For integer $m\ge 1$, denote the box by $\Lambda_m=[-m,m]^2$ and the semi-box by $\Lambda_m^+=[-m,m]\times[0,m]$. For integers $N>n\ge 1$, denote the annulus by $A(n,N)=\Lambda_N\setminus\Lambda_n$ and the semi-annulus by $A^+(n, N)=\Lambda_N^+\setminus\Lambda_n^+$. 
For $j\ge 1$ and a color pattern $\sigma\in \{0,1\}^j$, define $\LH_{\sigma}(n, N)$ (resp. $\LH_{\sigma}^+(n, N)$) to be the event that there are $j$ arms of the pattern $\sigma$ in the annulus $A(n, N)$ (resp. in the semi-annulus $A^+(n, N)$) connecting the inner boundary to the outer boundary. 
These probabilities should decay like a power in $N$ as $N\to\infty$: Suppose $\phi_{\Lambda_m}$ (resp. $\phi_{\Lambda_m^+}$) is the critical random-cluster probability measure on $\Lambda_{m}$ (resp. on  $\Lambda_m^+$) with $m\ge 2N$, we have, for fixed $n\ge 2j$, 
\[\phi_{\Lambda_m}[\LH_{\sigma}(n, N)]=N^{-\alpha_{\sigma}+o(1)}, \quad \phi_{\Lambda_m^+}[\LH^+_{\sigma}(n, N)]=N^{-\alpha^+_{\sigma}+o(1)},\quad\text{as }N\to\infty,\]
 for constant $\alpha_{\sigma}$ depending on $q$ and $\sigma$ and constant $\alpha_{\sigma}^+$ depending on $q$, $\sigma$ and the boundary conditions.
The exponents $\alpha_{\sigma}, \alpha_{\sigma}^+$ are called the interior critical arm exponent and the boundary critical arm exponent respectively.

In the case of percolation, Kesten \cite{KestenScalingRelationPercolation} proved the so-called scaling relations for the near-critical percolation. The 
existence and value of many exponents would follow from the existence and the value of critical $1$-arm exponent and $4$-arm exponent. 

In \cite{SchrammFirstSLE}, Oded Schramm  introduced Schramm Loewner Evolution as the candidate of the scaling limits of critical lattice models. In \cite{SmirnovPercolationConformalInvariance}, Smirnov proved the convergence of the interface in the critical percolation to $\SLE_6$, hence made it possible to calculate the value of the arm exponents of the critical percolation through $\SLE_6$. 
In \cite{SmirnovWernerCriticalExponents}, the authors explained that, in order to derive the arm exponents of critical percolation, one needs three inputs: 1. the convergence of the interface to $\SLE_6$; 2. the arm exponents of $\SLE_6$ (that is, the computation of asymptotic probabilities of certain events for $\SLE$ that mimic arm events); 
and 3. the quasi-multiplicativity of probabilities of arm events. 
The value of the arm exponents were computed using this strategy and the results of Smirnov in \cite{LawlerSchrammWernerExponent1}, \cite{LawlerSchrammWernerOneArmExponent}, \cite{SmirnovWernerCriticalExponents}.

In this paper, we will introduce the crossing events for $\SLE_{\kappa}$ which are the analogs of $\LH_{\sigma}(n, N)$ and $\LH_{\sigma}^+(n, N)$ defined above for random-cluster models. The parameters $\kappa$ and $q$ are related through~\eqref{eqn::kappa_q}. We will estimate the decay rate of these crossing events and derive
certain arm exponents of $\SLE_{\kappa}$ with $\kappa \in (4,8)$ in Theorems~\ref{thm::sle_boundaryarm} and~\ref{thm::sle_interiorarm}. 
In these theorems, we state the conclusion using the terminologies of the random-cluster model. The precise definition for $\SLE$ are sophisticated and we omit them from the introduction.  They will become clear in Sections~\ref{subsec::sle_boundary_statements} and~\ref{sec::sle_interior_arm}.

For FK-Ising model, the convergence of the interface to $\SLE_{16/3}$ is proved in \cite{ChelkakSmirnovIsing},\cite{CDCHKSConvergenceIsingSLE} and the quasi-multiplicativity is obtained in \cite{ChelkakDuminilHonglerCrossingprobaFKIsing}. Following the above strategy, it is then standard to deduce the arm exponents of the critical FK-Ising model, see \cite[Section~5]{WuAlternatingArmIsing}.  
For other random-cluster models that are conjectured to converge to $\SLE_{\kappa}$ as in~\eqref{eqn::kappa_q}, to derive the arm exponents, we are missing the convergence of the interface, but the arm exponents of the corresponding $\SLE_{\kappa}$ are given in Theorems~\ref{thm::sle_boundaryarm} and~\ref{thm::sle_interiorarm}, and the quasi-multiplicativity is announced in \cite[Section 1.3.3]{DuminilSidoraviciusTassionContinuityPhaseTransition}.
As long as the convergence of the interface is at hand, the values in Theorems~\ref{thm::sle_boundaryarm} and~\ref{thm::sle_interiorarm} are the arm exponents for the random-cluster model with $q$ related to $\kappa$ via~\eqref{eqn::kappa_q}.

\begin{theorem}\label{thm::sle_boundaryarm}
Fix $\kappa\in (4,8)$ and $j\ge 1$. We have the following six different patterns of boundary arm exponents of $\SLE_{\kappa}$. 
\begin{itemize}
\item Consider the wired boundary conditions $(\underline{11})$ and the pattern $\sigma=(010\cdots10)$ with length $2j\!-\!1$. The corresponding boundary arm exponents are given by
\begin{equation}\label{eqn::sle_boundary_alphaodd}
\alpha^+_{2j-1}=j(4j+4-\kappa)/\kappa.
\end{equation}
\item Consider the wired boundary conditions $(\underline{11})$ and the pattern $\sigma=(010\cdots1)$ with length $2j$. The corresponding boundary arm exponents are given by
\begin{equation}\label{eqn::sle_boundary_betaeven}
\beta^+_{2j}=j(4j+\kappa-4)/\kappa.
\end{equation}
\item Consider the wired boundary conditions $(\underline{11})$ and the pattern $\sigma=(101\cdots01)$ with length $2j\!+\!1$. The corresponding boundary arm exponents are given by
\begin{equation}\label{eqn::sle_boundary_gammaodd}
\gamma^+_{2j+1}=(j+1)(4j+3\kappa-16)/\kappa+(\kappa-6)(\kappa-8)/(2\kappa).
\end{equation}
\item Consider the free/wired boundary conditions $(\underline{01})$ and the pattern $\sigma=(10\cdots10)$ with length $2j$. The corresponding boundary arm exponents are given by
\begin{equation}\label{eqn::sle_boundary_alphaeven}
\alpha^+_{2j}=j(4j+8-\kappa)/\kappa.
\end{equation}
\item Consider the free/wired boundary conditions $(\underline{01})$ and the pattern $\sigma=(10\cdots101)$ with length $2j\!-\!1$. The corresponding boundary arm exponents are given by
\begin{equation}\label{eqn::sle_boundary_betaodd}
\beta^+_{2j-1}=j(4j+\kappa-8)/\kappa.
\end{equation}
\item Consider the free/wired boundary conditions $(\underline{01})$ and the pattern $\sigma=(0101\cdots01)$ with length $2j$. The corresponding boundary arm exponents are given by
\begin{equation}\label{eqn::sle_boundary_gammaeven}
\gamma^+_{2j}=j(4j+3\kappa-16)/\kappa+(\kappa-4)(\kappa-6)/(2\kappa).
\end{equation}
\end{itemize} 
\end{theorem}

We list the formulae for six different patterns in Theorem~\ref{thm::sle_boundaryarm} for completeness. Not all of them are new: The formulae~\eqref{eqn::sle_boundary_alphaodd}, \eqref{eqn::sle_boundary_betaeven},  \eqref{eqn::sle_boundary_alphaeven}, and~\eqref{eqn::sle_boundary_betaodd} were obtained in \cite{WuZhanSLEBoundaryArmExponents}. The novelty part of this theorem are the formulae~\eqref{eqn::sle_boundary_gammaodd} and~\eqref{eqn::sle_boundary_gammaeven}. We will prove these two formulae in Sections~\ref{subsec::sle_boundary_gammaodd_proba} and~\ref{subsec::sle_boundary_gammaeven_proba} using~\eqref{eqn::sle_boundary_betaodd} and~\eqref{eqn::sle_boundary_betaeven}.

\begin{theorem}\label{thm::fkising_boundaryarm}
For the critical planar FK-Ising model on $\Z^2$, we have six different patterns of boundary arm exponents as in Theorem \ref{thm::sle_boundaryarm} taking $\kappa=16/3$. 
%We list several of them here:
%\[\alpha_1^+=1/2, \quad \alpha^+_2=5/4,\quad \beta_1^+=1/4, \quad \beta_2^+=1, \quad \gamma_2^+=2/3, \quad \gamma_3^+=5/3.\]
\end{theorem}

\begin{theorem}\label{thm::sle_interiorarm}
Fix $\kappa\in (4,8)$ and $j\ge 1$. We have the following three different patterns of interior arm exponents of $\SLE_{\kappa}$. 
\begin{itemize}
\item Let $\sigma=(10\cdots 10)$ with length $2j$. The corresponding interior arm exponents are given by 
\begin{equation}\label{eqn::sle_interior_alpha}
\alpha_{2j}= \left(16j^2-(\kappa-4)^2\right)/(8\kappa).
\end{equation}
\item Let $\sigma=(10\cdots 101)$ with length $2j+1$. The corresponding interior arm exponents are given by 
\begin{equation}\label{eqn::sle_interior_beta}
\beta_{2j+1}= j(2j+\kappa-4)/\kappa.
\end{equation}
\item Let $\sigma=(0110\cdots 10)$\footnote{The pattern $\sigma=(0110\cdots 10)$ with length $2j+2$ starts with $01$ and then it is followed by $j$ pairs of $10$} with length $2j+2$. The corresponding interior arm exponents are given by 
\begin{equation}\label{eqn::sle_interior_gamma}
\gamma_{2j+2}= \left(4(2j+\kappa-4)^2-(\kappa-4)^2\right)/(8\kappa).
\end{equation}
\end{itemize}
\end{theorem}

The value of $\alpha_2$ in~\eqref{eqn::sle_interior_alpha} were obtained in \cite{BeffaraDimension}. The formula~\eqref{eqn::sle_interior_alpha} were proved for $\SLE_{\kappa}$ with $\kappa\le 4$ in \cite[Section 4]{WuAlternatingArmIsing}. But the formulae~\eqref{eqn::sle_interior_beta} and~\eqref{eqn::sle_interior_gamma} only make sense for $\kappa\in (4,8)$. We will prove all the three formulae in Section~\ref{sec::sle_interior_arm} using~\eqref{eqn::sle_boundary_alphaeven}, \eqref{eqn::sle_boundary_betaodd} and~\eqref{eqn::sle_boundary_gammaeven}.

%\begin{figure}
%\includegraphics[width=0.32\textwidth]{figures/interior_alpha}
%\includegraphics[width=0.32\textwidth]{figures/interior_beta}
%\includegraphics[width=0.32\textwidth]{figures/interior_gamma}
%\caption{\label{fig::interior_arms} Three different patterns of interior arm exponents in Theorem \ref{thm::sle_interiorarm}. From left to right: $\alpha_6$: (101010); $\beta_7$: (1010110); $\gamma_8$: (10101100).}
%\end{figure}

%\begin{figure}[ht!]
%\begin{subfigure}[b]{0.32\textwidth}
%\begin{center}
%\includegraphics[width=\textwidth]{figures/interior_alpha}
%\end{center}
%\caption{$\alpha_6$: (101010).}
%\end{subfigure}
%\begin{subfigure}[b]{0.32\textwidth}
%\begin{center}
%\includegraphics[width=\textwidth]{figures/interior_beta}
%\end{center}
%\caption{$\beta_7$: (1010110).}
%\end{subfigure}
%\begin{subfigure}[b]{0.32\textwidth}
%\begin{center}
%\includegraphics[width=\textwidth]{figures/interior_gamma}
%\end{center}
%\caption{$\gamma_8$: (10101100).}
%\end{subfigure}
%\caption{\label{fig::interior_arms} The three different patterns of interior arm exponents in Theorem \ref{thm::sle_interiorarm}.}
%\end{figure}

\begin{theorem}\label{thm::fkising_interiorarm}
For the critical planar FK-Ising model on $\Z^2$, we have three different patterns of interior arm exponents as in Theorem \ref{thm::sle_interiorarm} taking $\kappa=16/3$. 
%We list several of them here: 
%\[\alpha_2=1/3, \quad \beta_3=5/8,\quad \gamma_4=1,\quad \alpha_4=35/24,\quad \beta_5=2,\quad \alpha_6=10/3.\]
\end{theorem}
%\begin{figure}[ht!]
%\begin{center}\includegraphics[width=0.5\textwidth]{figures/arm_exponents_functions}\end{center}
%{\caption{
%The $x$-axis corresponds to $\kappa\in (4,8)$. The blue lines are $\alpha_2, \alpha_4, \alpha_6, \alpha_8$ as functions of $\kappa$. The red lines are $\beta_3, \beta_5, \beta_7$ as functions of $\kappa$. The yellow lines are $\gamma_4, \gamma_6, \gamma_8$ as functions of $\kappa$.}}
%\end{figure}

In Theorem~\ref{thm::fkising_interiorarm}, we derive various patterns of the the arm exponents for the critical planar FK-Ising model, but we do not address the case of one-arm exponent. 
The one-arm exponent for FK-Ising model equals $1/8$, derived in \cite{WuTheoryToeplitzDeterminantsSpinCorrelation}, see also the discussion in \cite[Section 6]{GarbanWuDustAnalysisFKIsing}. 
In general, the one-arm exponent of random-cluster model is conjecture to be $\tilde{\alpha}_1:=(8-\kappa)(3\kappa-8)/(32\kappa)$ where $\kappa$ and $q$ are related via~\eqref{eqn::kappa_q}. 
The value of $\tilde{\alpha}_1$ is the same as the one-arm exponent of $\SLE_{\kappa}(\kappa-6)$ derived in \cite{SchrammSheffieldWilsonConformalRadii}.

\begin{remark}
In Theorem \ref{thm::sle_boundaryarm}, if we set $\kappa=6$ then we find all the six formulae have the same expression:
\[\alpha^+_j=\beta^+_j=\gamma^+_j=\frac{j(j+1)}{6},\]
which are the boundary arm exponents for the critical percolation. The reason is that the boundary arm exponents for the critical percolation are independent of boundary conditions and are the same over all patterns of any given length. In Theorem \ref{thm::sle_interiorarm}, if we set $\kappa=6$ then we find all the three formulae have the same expression
\[\alpha_{2j}=\gamma_{2j}=\frac{(2j)^2-1}{12},\quad \beta_{2j+1}=\frac{(2j+1)^2-1}{12},\] which are the interior arm exponents for the critical percolation. The reason is that the interior arm exponents for the critical percolation are the same over all patterns of any given length, as long as they are polychromatic, i.e. $\sigma$ is not constant. These arm exponents for the critical site percolation on the triangular lattice were derived in \cite{LawlerSchrammWernerExponent1},\cite{SmirnovWernerCriticalExponents}. To see that the exponents are the same for all patterns, the authors used ``color switching trick" which is only valid for the triangular lattice.
\end{remark}

\begin{remark} We point out some interesting facts in the formulae of Theorems \ref{thm::sle_boundaryarm} and \ref{thm::sle_interiorarm}: 
\[\beta_2^+=1,\quad \beta_3^+=2,\quad \beta_5=2, \quad \forall \kappa\in (4,8).\]
All these three exponents are supposed to be universal for random-cluster model with $q\in [1,4)$.  They are proved for $q=2$ in \cite[Corollary~1.5]{ChelkakDuminilHonglerCrossingprobaFKIsing}. See discussion in \cite[Section 1.3.3]{DuminilSidoraviciusTassionContinuityPhaseTransition} for other random-cluster models.
\end{remark}
\smallbreak

\noindent\textbf{Relation to previous works.}

The proof for Theorems~\ref{thm::fkising_boundaryarm} and~\ref{thm::fkising_interiorarm} from Theorems~\ref{thm::sle_boundaryarm} and~\ref{thm::sle_interiorarm} are standard and we refer interested readers to \cite{SmirnovWernerCriticalExponents} or \cite[Section 5]{WuAlternatingArmIsing}. In \cite{WuAlternatingArmIsing}, the author derived results similar to Theorems~\ref{thm::fkising_boundaryarm} and~\ref{thm::fkising_interiorarm} for the critical planar Ising model where the interface converges to $\SLE_3$.

The 2-arm exponents $\alpha_2$ is related to the Hausdorff dimension of SLE which is $2-\alpha_2$. This dimension was obtained in \cite{BeffaraDimension}. The 3-arm exponents $\beta_3$ is related to the Hausdorff dimension of the frontier of SLE which is $2-\beta_3$. This dimension is the same as the dimension of $\SLE_{16/\kappa}$ by duality. 
The 4-arm exponent $\alpha_4$ is related to the Hausdorff dimension of the double points of SLE which is $2-\alpha_4$. This dimension was obtained in \cite[Theorem~1.1]{MillerWuSLEIntersection}.
The 4-arm exponent $\gamma_4$ is related to the Hausdorff dimension of the cut points of SLE which is $2-\gamma_4$. This dimension was obtained in \cite[Theorem~1.2]{MillerWuSLEIntersection}.
But, in general, our results about the arm exponents do neither imply nor are implied by the conclusions on Hausdorff dimensions.

The formulae \eqref{eqn::sle_boundary_alphaodd} and  \eqref{eqn::sle_interior_alpha} were predicted by KPZ in \cite[Eq.~(11.44), Eq.~(11.45)]{DuplantierFractalGeometry}, and our work confirms those predictions. 
\medbreak
To end the introduction, let us mention the arm exponents of the patterns which are not listed in Theorems~\ref{thm::sle_boundaryarm} or~\ref{thm::sle_interiorarm}. 
In general, for the patterns are not listed in Theorems~\ref{thm::sle_boundaryarm} or~\ref{thm::sle_interiorarm}, it is not clear how to relate them to the arm-exponents of $\SLE$, and hence we are not able to derive their values. 

The simplest case is the monochromatic arm exponents---the pattern $\sigma$ is contant. It is proved in \cite{BeffaraNolinMonochromaticArmPercolation} that the monochromatic interior arm exponents for the critical percolation are distinct from the polychromatic ones, and their values are still unknown. 

But it is still possible to derive certain arm exponents by closer analysis on the discrete models. One example is the interior six-arm exponent of FK-Ising model with the pattern $\sigma=(100100)$. Note that this  pattern is not included in Theorem~\ref{thm::sle_interiorarm} and we do not know how to relate it to arm exponents of $\SLE_{16/3}$. By a private communication with Vincent Tassion, we believe its value is the same as the interior four-arm exponent of Ising model which equals $21/8$. The proof is not written yet, and the argument involves Edward-Sokal coupling and a sophisticated ``dust analysis".

%%%%
\section{Preliminaries on SLE}
\label{sec::sle_preliminaries}
%%%%
\noindent\textbf{Notation.} For functions $f$ and $g$, we denote by $f\lesssim g$ if $f/g$ is bounded from above by universal finite constant, by $f\gtrsim g$ if $f/g$ is bounded from below by universal positive constant, and by $f\asymp g$ if $f\lesssim g$ and $f\gtrsim g$. We denote by  
\[f(\eps)=g(\eps)^{1+o(1)}\quad \text{if}\quad \lim_{\eps\to 0}\frac{\log f(\eps)}{\log g(\eps)}=1.\]

\noindent For $z\in\C, r>0$, we denote
$B(z, r)=\{w\in\C: |w-z|<r\}$. We denote the unit disc $B(0,1)$ by $\U$. 

\noindent For two subsets $A, B\subset\C$, we denote $ \dist(A, B)=\inf\{|x-y|: x\in A, y\in B\}$.
\medbreak
\noindent\textbf{$\HH$-hull and Loewner chain} 
We call a compact subset $K$ of $\overline{\HH}$ an $\HH$-hull if $\HH\setminus K$ is simply connected. Riemann's Mapping Theorem and the Reflection Principle assert that there exists a unique conformal map $g_K$ from $\HH\setminus K$ onto $\HH$ such that
$\lim_{|z|\to\infty}|g_K(z)-z|=0$.
We call such $g_K$ the conformal map from $\HH\setminus K$ onto $\HH$ normalized at $\infty$.

The following two lemmas estimate the image of balls under conformal maps. Lemma~\ref{lem::image_insideball} is a standard estimate using the Koebe 1/4 theorem. 
\begin{lemma}
\label{lem::extremallength_argument}
\cite[Lemma 2.1]{WuAlternatingArmIsing}
Fix $x>0$ and $\eps>0$.
Let $K$ be an $\HH$-hull and let $g_K$ be the conformal map from $\HH\setminus K$ onto $\HH$ normalized at $\infty$. Assume that
$x>\max(K\cap\R)$.
Denote by $\gamma$ the connected component of $\HH\cap (\partial B(x,\eps)\setminus K)$ whose closure contains $x+\eps$. Then $g_K(\gamma)$ is contained in the ball with center $g_K(x+\eps)$ and radius $3(g_K(x+3\eps)-g_K(x+\eps))$, hence it is also contained in the ball with center $g_K(x+3\eps)$ and radius $8\eps g_K'(x+3\eps)$.
\end{lemma}

\begin{lemma}\label{lem::image_insideball}
Fix $z\in\overline{\HH}$ and $\eps>0$. Let $K$ be an $\HH$-hull and let $g_K$ be the conformal map from $\HH\setminus K$ onto $\HH$ normalized at $\infty$. Assume that
$\dist(K, z)\ge 16\eps$.
Then $g_K(B(z,\eps))$ is contained in the ball with center $g_K(z)$ and radius $4\eps |g_K'(z)|$. 
\end{lemma}

Consider the family of conformal maps $(g_{t}, t\ge 0)$ obtained by solving the Loewner equation: for each $z\in\mathbb{H}$,
\begin{equation*}
\partial_{t}{g}_{t}(z)=\frac{2}{g_{t}(z)-W_{t}}, \quad g_{0}(z)=z,
\end{equation*}
where $(W_t, t\ge 0)$ is a one-dimensional continuous function which we call the driving function. Let $T_z$ be the swallowing time of $z$ defined as $\sup\{t\ge 0: \min_{s\in[0,t]}|g_{s}(z)-W_{s}|>0\}$.
Let $K_{t}:=\overline{\{z\in\mathbb{H}: T_{z}\le t\}}$. Then $g_{t}$ is the unique conformal map from $H_{t}:=\mathbb{H}\backslash K_{t}$ onto $\mathbb{H}$ normalized at $\infty$. A Loewner chain is the collection of $\HH$-hulls $(K_{t}, t\ge 0)$ associated with the family of conformal maps $(g_t, t\ge 0)$.

Here we discuss a little about the evolution of a point $y\in\R$ under $g_t$. We assume $y\le 0$. There are two possibilities: if $y$ is not swallowed by $K_t$, then we define $Y_t=g_t(y)$; if $y$ is swallowed by $K_t$, then we define $Y_t$ to the be image of the leftmost of point of $K_t\cap\R$ under $g_t$. Suppose that $(K_t, t\ge 0)$ is generated by a continuous path $(\eta(t), t\ge 0)$ and that the Lebesgue measure of $\eta[0,\infty]\cap\R$ is zero. Then the process $Y_t$ is uniquely characterized by the following equation: 
\[Y_t=y+\int_0^t \frac{2ds}{Y_s-W_s},\quad Y_t\le W_t,\quad \forall t\ge 0.\] 
Although $g_t(y)$ is only well-defined when $y$ is not swallowed by $K_t$, we still write $g_t(y)$ for all time $t$. When $y$ IS swallowed by $K_t$, the notation $g_t(y)$ stands for the process $Y_t$. 
\medbreak
\noindent\textbf{SLE processes} 
An $\SLE_{\kappa}$ is the random Loewner chain $(K_{t}, t\ge 0)$ driven by $W_t=\sqrt{\kappa}B_t$ where $(B_t, t\ge 0)$ is a standard one-dimensional Brownian motion.
In \cite{RohdeSchrammSLEBasicProperty}, the authors prove that $(K_{t}, t\ge 0)$ is almost surely generated by a continuous transient curve, i.e. there almost surely exists a continuous curve $\eta$ such that for each $t\ge 0$, $H_{t}$ is the unbounded connected component of $\mathbb{H}\backslash\eta[0,t]$ and that $\lim_{t\to\infty}|\eta(t)|=\infty$.

Next, we define $\SLE_{\kappa}(\rho^{L}; \rho^{R}; \rho^I)$ process with three force points $(x^{L}; x^{R}; z)$ where $\rho^L, \rho^R, \rho^I\in \R$ and $x^{L}\le 0\le x^{R}$ and $z\in\HH$. It is the Loewner chain driven by $W_{t}$ which is the solution to the following systems of SDEs:
\[dW_{t}=\sqrt{\kappa}dB_{t}+\frac{\rho^{L} dt}{W_{t}-V_{t}^{L}}+\frac{\rho^{R} dt}{W_{t}-V_{t}^{R}}+\re{\frac{\rho^I dt}{W_t-V^I_t}}, \quad W_{0}=0;\]
\[dV^{L}_{t}=\frac{2dt}{V^{L}_{t}-W_{t}}, \quad V^{L}_{0}=x^{L}; \quad dV^{R}_{t}=\frac{2dt}{V^{R}_{t}-W_{t}}, \quad V^{R}_{0}=x^{R};\]
\[dV^I_t=\frac{2dt}{V^I_t-W_t}, \quad V^I_0=z.\]

The solution exists up to the first time that $W$ hits $V^L$, $V^R$ or $V^I$.
Suppose $\rho^I=0$, 
when $\rho^L>-2$ and $\rho^R>-2$, the solution exists for all times, and the corresponding Loewner chain is almost surely generated by a continuous transient curve \cite[Section 2 and Theorem 1.3]{MillerSheffieldIG1}. When $\rho^I\neq 0$ and $\rho^L>-2, \rho^R>-2$, the solution exists up to the first time that $z$ is swallowed, and the corresponding Loewner chain is almost surely generated by a continuous curve \cite[Section 2.1]{MillerSheffieldIG4}. If $\rho^I=0, \rho^L=0$ or $\rho^R=0$, they will be omitted from the notation.

Fix $\rho^I=0$ and $\rho^L>-2$. There are two special values of $\rho^R$: $\kappa/2-2$ and $\kappa/2-4$. When $\rho^R\ge \kappa/2-2$, the curve never hits the interval $[x^R,\infty)$. When $\rho^R\in (\kappa/2-4, \kappa/2-2)$, the curve accumulates at a point in $(x^R, \infty)$ at finite time. 
When $\rho^R\le \kappa/2-4$, the curve converges almost surely to $x^R$ at finite time, see \cite[Lemma 15]{DubedatSLEDuality}.

The $\SLE$ processes satisfy the \textit{Domain Markov Property}: Let $\eta$ be an $\SLE_{\kappa}(\rho^L, \rho^R; \rho^I)$ process with force points $(x^L, x^R; z)$. Suppose that $\tau$ is any stopping time (before $z$ is swallowed), then the image of $\eta[\tau, \infty)$ under $g_{\tau}-W_{\tau}$ has the same law as an $\SLE_{\kappa}(\rho^L,\rho^R; \rho^I)$ process with force points $(V^L_{\tau}; V^R_{\tau}; V^I_{\tau})$.

From Girsanov Theorem, it follows that the law of an $\SLE_{\kappa}(\rho^L;\rho^R;\rho^I)$ process can be constructed by reweighting the law of an ordinary $\SLE_{\kappa}$, see \cite[Theorem 6]{SchrammWilsonSLECoordinatechanges}.
\begin{lemma}\label{lem::sle_mart}
Suppose $x^L<0<x^R$ and $z\in \HH$, define
\begin{align*}
M_t(x^L;x^R)=&g_t'(x^L)^{\rho^L(\rho^L+4-\kappa)/(4\kappa)}(W_t-g_t(x^L))^{\rho^L/\kappa} \\
&\times g_t'(x^R)^{\rho^R(\rho^R+4-\kappa)/(4\kappa)}(g_t(x^R)-W_t)^{\rho^R/\kappa}\\
&\times (g_t(x^R)-g_t(x^L))^{\rho^L\rho^R/(2\kappa)};\\
M_t(z)=&|g_t'(z)|^{\rho^I(\rho^I+8-2\kappa)/(8\kappa)}(\im{g_t(z)})^{(\rho^I)^2/(8\kappa)}|g_t(z)-W_t|^{\rho^I/\kappa}.
\end{align*}
Then $M(x^L;x^R)$ is a local martingale for $\SLE_{\kappa}$ and the law of $\SLE_{\kappa}$ weighted by $M(x^L;x^R)$ (up to the first time that $W$ hits one of the force points) equals the law of $\SLE_{\kappa}(\rho^L;\rho^R)$ with force points $(x^L; x^R)$. Also, $M(z)$ is a local martingale for $\SLE_{\kappa}$ and the law of $\SLE_{\kappa}$ weighted by $M(z)$ (up to the first time that $z$ is swallowed) equals the law of $\SLE_{\kappa}(\rho^I)$ with force point $z$. 
\end{lemma}

The following lemmas are technical. They give lower bounds on the probability for an SLE curve to behave nicely. It is important later that the lower bound on the probability is uniform over the location of the force points. 
\begin{lemma}\label{lem::sle_goodbehavior}
\cite[Lemma 2.4]{MillerWuSLEIntersection}.
Suppose that $\eta$ is an $\SLE_{\kappa}(\rho^L;\rho^R)$ process in $\HH$ from 0 to $\infty$ with force points $x^L\le 0\le x^R$. Fix $\kappa\in (0,8)$ and $\rho^L, \rho^R>(-2)\vee(\kappa/2-4)$. 
Fix $\delta\in (0,1/2)$ and define the stopping time $S_1=\inf\{t: \eta(t)\in\partial B(i,\delta)\}$. 
Denote by $U(\delta)$ the $\delta$-neighborhood of the segment connecting $0$ to $i$. Define the stopping time $S_2=\inf\{t: \eta(t)\not\in U(\delta)\}$. Then there exists $p_0=p_0(\delta)>0$ such that  
%\begin{equation}\label{eqn::sle_tube_estimate1}
$\PP[S_1<S_2]\ge p_0$.
%\end{equation}
We emphasize that $p_0$ may depend on $\kappa, \rho^L, \rho^R$ or $\delta$, but it is uniform over $x^L, x^R$.
\end{lemma}

\begin{lemma}\label{lem::sle_goodbehavior2}
Suppose that $\eta$ is an $\SLE_{\kappa}(\rho^L;\rho^R)$ process in $\HH$ from 0 to $\infty$ with force points $x^L\le 0\le x^R$. Fix $\kappa\in (0,8)$ and $\rho^L, \rho^R>(-2)\vee(\kappa/2-4)$. 
Let $\xi$ be the first time that $\eta$ exits $B(0,1)$. 
Then there exists $ q_0>0$ such that, for $\delta$ small enough,
%\begin{equation}\label{eqn::sle_tube_estimate2}
\[\PP[\im\eta(\xi)\ge\delta, \dist(\eta[0,\xi], \eps)\ge \eps/4]\ge q_0.\]
%\end{equation} 
We emphasize that $q_0$ may depend on $\kappa, \rho^L$ or $\rho^R$, but it is uniform over $x^L, x^R$ and $\eps$. 
\end{lemma}
\begin{proof}
Consider the event $\{\eta\text{ hits } B(1,\delta)\}$, since $\kappa<8, \rho^R>(-2)\vee(\kappa/2-4)$, there exist a function $ q(\delta)\to 0$ as $\delta\to 0$ such that (see \cite[Lemma 6.5]{WuAlternatingArmIsing})
\[\PP[\eta\text{ hits } B(1,\delta)]\le q(\delta).\]
It is important that $q(\delta)$ is uniform over $x^L, x^R$. 
Note that 
\begin{align*}
\PP[\im\eta(\xi)\ge\delta, \dist(\eta[0,\xi], \eps)\ge \eps/4]
&\ge 1-\PP[\eta\text{ hits } B(\eps, \eps/4)]-\PP[\im\eta(\xi)\le\delta]\\
&\ge 1-q(1/4)-2q(\delta).
\end{align*}
Since $q(\delta)\to 0$ as $\delta\to 0$, this implies the conclusion. 
\end{proof}

%%%%%
\section{Proof of Theorem \ref{thm::sle_boundaryarm}}
\label{sec::sle_boundary_arm}
%%%%%
In this section, we will first define the crossing events for $\SLE$ which correspond to the different cases in Theorem~\ref{thm::sle_boundaryarm}. The formulae \eqref{eqn::sle_boundary_alphaodd}, \eqref{eqn::sle_boundary_betaeven},  \eqref{eqn::sle_boundary_alphaeven}, \eqref{eqn::sle_boundary_betaodd} were derived \cite{WuZhanSLEBoundaryArmExponents}, and we will use these results to prove the formulae \eqref{eqn::sle_boundary_gammaodd}, \eqref{eqn::sle_boundary_gammaeven}, hence completes the proof of Theorem~\ref{thm::sle_boundaryarm}. Set $\alpha_0^+=\beta_0^+=\gamma_0^+=\gamma_1^+=0$ and $\alpha_j^+, \beta_j^+, \gamma_j^+$ are the numerical values listed in Theorem~\ref{thm::sle_boundaryarm}.

\subsection{Definitions and Statements}
\label{subsec::sle_boundary_statements}

Fix $\kappa\in (4,8)$ and let $\eta$ be an $\SLE_{\kappa}$ in $\HH$ from 0 to $\infty$. Suppose $y\le 0<\eps\le x$ and let $T_x$ be the swallowing time of $x$ which is almost surely finite since $\kappa\in (4,8)$. We are interested in the crossings of $\eta$ between the ball $B(x,\eps)$ and the interval $(-\infty, y)$. 

We first define the crossing events $\LH^{\alpha}_{2j-1}$, $\LH^{\beta}_{2j}$ and $\LH^{\gamma}_{2j-1}$ for $j\ge 1$ which correspond to \eqref{eqn::sle_boundary_alphaodd}, \eqref{eqn::sle_boundary_betaeven}, \eqref{eqn::sle_boundary_gammaodd}. 
Suppose that $y\le 0<\eps\le u\le x$ and let $T_u$ be the swallowing time of $u$. Set $\tau_0=\sigma_0=0$. Let $\tau_1$ be the first time that $\eta$ hits the ball $B(x,\eps)$ and let $\sigma_1$ be the first time after $\tau_1$ that $\eta$ hits $(-\infty, y)$. For $j\ge 1$, let $\tau_j$ be the first time after $\sigma_{j-1}$ that $\eta$ hits the connected component of $\partial B(x,\eps)\setminus \eta[0,\sigma_{j-1}]$  containing $x+\eps$ and let $\sigma_j$ be the first time after $\tau_j$ that $\eta$ hits $(-\infty, y)$. Define
\[\LH^{\alpha}_{2j-1}(\eps, x, y)=\{\tau_j<T_x\},\quad \LH^{\beta}_{2j}(\eps, x, y)=\{\sigma_j<T_x\},\] 
\[\LH^{\gamma}_{2j+1}(\eps, x, y, u)=\{\tau_1<T_u, \sigma_j<T_x\}.\]

Imagine that $\eta$ is the interface in the FK-Ising model and the boundary conditions are free $(\underline{0})$ on $\R_-$ and are wired $(\underline{1})$ on $\R_+$, then the event $\LH^{\alpha}_{2j-1}(\eps, 1, 0)$ interprets that there are $2j-1$ arms going between $\partial B(1,\eps)$ and $\partial B(1,1/2)$ of the pattern $(010\cdots10)$ clockwise. 
The event $\LH^{\beta}_{2j}(\eps, 1, 0)$ interprets that there are $2j$ arms going between $\partial B(1,\eps)$ and $\partial B(1,1/2)$ of the pattern $(01\cdots 01)$ clockwise. 
And the event $\LH^{\gamma}_{2j+1}(\eps, 1, 0, 1/2)$ interprets that there are $2j+1$ arms going between $\partial B(1,\eps)$ and $\partial B(1,1/2)$ of the pattern $(10\cdots101)$ clockwise.  
See Fig.~\ref{fig::boundary_arms11}.

\begin{figure}
\begin{center}
\includegraphics[width=0.6\textwidth]{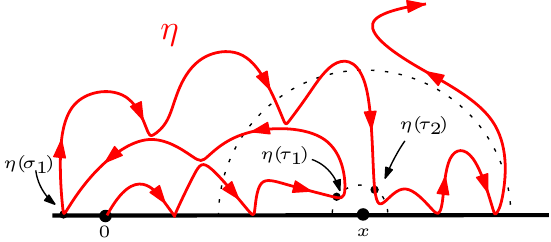}
\end{center}
\includegraphics[width=0.32\textwidth]{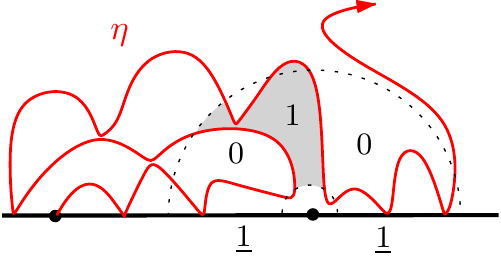}
\includegraphics[width=0.32\textwidth]{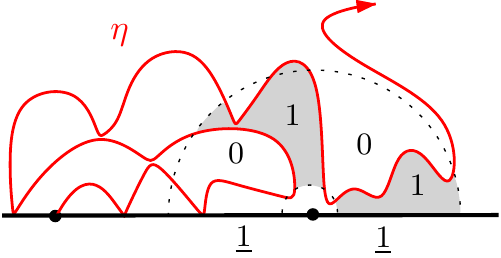}
\includegraphics[width=0.32\textwidth]{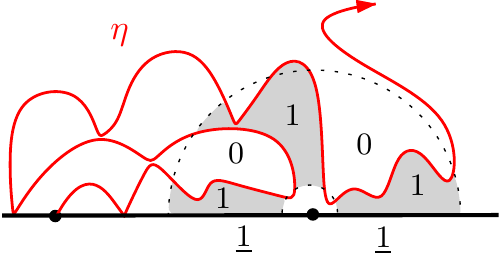}
\caption{\label{fig::boundary_arms11}
In the panel on top, the curve $\eta$ is an $\SLE_{\kappa}$ in $\HH$ from $0$ to $\infty$. Let $\tau_1$ be the first time that $\eta$ hits $B(x,\eps)$. Let $\sigma_1$ be the first time after $\tau_1$ that $\eta$ hits $(-\infty, 0)$. Let $\tau_2$ be the first time after $\sigma_1$ that $\eta$ hits the connected component of $\partial B(x,\eps)\setminus\eta[0,\sigma_1]$ containing $x+\eps$.  
In the panel on bottom left, this figure indicates the event 
$\{\tau_2<T_x\}$ which corresponds to the case of three arms 
$\alpha_3^+$ of the pattern $(010)$. 
In the panel on bottom middle, this figure indicates the event 
$\{\sigma_2<T_x\}$ which corresponds to the case of four arms 
$\beta_4^+$ of the pattern $(0101)$.
In the panel on bottom right, this figure indicates the event 
$\{\tau_1<T_u, \sigma_2<T_x\}$ which corresponds to the case of five arms $\gamma_5^+$ of the pattern $(10101)$. }
\end{figure}

It is proved in \cite[Theorems 1.1, 1.2]{WuZhanSLEBoundaryArmExponents} that, for any $y\le 0<\eps\le x$ and $j\ge 1$, we have
\begin{equation}\label{eqn::sle_boundary_alphaodd_proba}
\PP[\LH^{\alpha}_{2j-1}(\eps, x, y)]\asymp \left(\frac{x}{x-y}\right)^{\alpha_{2j-2}^+}\left(\frac{\eps}{x}\right)^{\alpha_{2j-1}^+},
\end{equation}
\begin{equation}\label{eqn::sle_boundary_betaeven_proba}
\PP\left[\LH^{\beta}_{2j}(\eps, x, y)\right]\asymp \left(\frac{x}{x-y}\right)^{\beta_{2j-1}^+}\left(\frac{\eps}{x}\right)^{\beta_{2j}^+},
\end{equation}
where the constants in $\asymp$ are uniform over $x,y,\eps$.
In particular, for fixed $\delta>0$, we have, for $\delta\le x\le 1/\delta$, and $-1/\delta\le y\le 0$,
\begin{equation}\label{eqn::sle_boundary_meaning1}
\PP[\LH^{\alpha}_{2j-1}(\eps, x, y)]\asymp \eps^{\alpha_{2j-1}^+},\quad\PP\left[\LH^{\beta}_{2j}(\eps, x, y)\right]\asymp\eps^{\beta_{2j}^+},
\end{equation}
where the constants in $\asymp$ are uniform over $\eps$. 
The readers may take~\eqref{eqn::sle_boundary_meaning1} as the definition of~\eqref{eqn::sle_boundary_alphaodd} and~\eqref{eqn::sle_boundary_betaeven}.
In Section~\ref{subsec::sle_boundary_gammaodd_proba}, we will prove the following estimate on $\LH_{2j+1}^{\gamma}$. 

\begin{proposition}\label{prop::sle_boundary_gammaodd_proba}
Fix some $\delta>0$ small, define 
\[\LF=\{\tau_1<T_u, \eta[0,\tau_1]\subset B(0,1/\delta), \dist(\eta[0,\tau_1], [x-\eps, x+3\eps])\ge\delta\eps\}.\]
Then, for $\delta\le u\le x/2\le 1/\delta$, and $-1/\delta\le y\le 0$,
\begin{equation}\label{eqn::sle_boundary_gammaodd_proba}
\PP\left[\LH^{\gamma}_{2j+1}(\eps, x, y, u)\cap\LF\right]\asymp \eps^{\gamma_{2j+1}^+}, \end{equation}
where the constants in $\asymp$ are uniform over $\eps$.  The readers may take~\eqref{eqn::sle_boundary_gammaodd_proba} as the definition of~\eqref{eqn::sle_boundary_gammaodd}.
\end{proposition}

\smallbreak
Next, we define the crossing events $\LH^{\alpha}_{2j}, \LH^{\beta}_{2j+1}, \LH^{\gamma}_{2j+2}$ for $j\ge 0$ which correspond to \eqref{eqn::sle_boundary_alphaeven}, \eqref{eqn::sle_boundary_betaodd}, \eqref{eqn::sle_boundary_gammaeven}. Suppose that $y\le u\le -\eps\le\eps\le x$ and let $T_u$ be the swallowing time of $u$. 
Set $\tau_0=\sigma_0=0$. 
We emphasize that we will change the definition of the stopping times in the following. 
Let $\sigma_1$ be the first time that $\eta$ hits $(-\infty, y)$ and $\tau_1$ be the first time after $\sigma_1$ that $\eta$ hits the connected component of $\partial B(x,\eps)\setminus\eta[0,\sigma_1]$ containing $x+\eps$. For $j\ge 1$, let $\sigma_j$ be the first time after $\tau_{j-1}$ that $\eta$  hits $(-\infty, y)$ and $\tau_j$ be the first time after $\sigma_j$ that $\eta$ hits the connected component of $\partial B(x,\eps)\setminus \eta[0,\sigma_j]$ containing $x+\eps$. Define 
\[\LH^{\alpha}_{2j}(\eps, x, y)=\{\tau_j<T_x\},\quad \LH^{\beta}_{2j+1}(\eps, x, y)=\{\sigma_{j+1}<T_x\},\] \[\LH^{\gamma}_{2j+2}(\eps, x, y, u)=\{\sigma_1=T_u, \sigma_{j+1}<T_x\}.\]

Imagine that $\eta$ is the interface in the FK-Ising model and the boundary conditions are free $(\underline{0})$ on $\R_-$ and are wired $(\underline{1})$ on $\R_+$, then the event $\LH^{\alpha}_{2j}(\eps, \eps, -2)$ interprets that there are $2j$ arms going between $\partial B(0,4\eps)$ and $\partial B(0,1/2)$ of the pattern $(10\cdots 10)$ clockwise.
The event $\LH^{\beta}_{2j+1}(\eps, \eps, -2)$ interprets that there are $2j+1$ arms going 
between $\partial B(0,4\eps)$ and $\partial B(0,1/2)$ of the pattern $(10\cdots 101)$ clockwise.
And the event $\LH^{\gamma}_{2j+2}(\eps, \eps, -2, -\eps)$ interprets that there are $2j+2$ arms going 
between $\partial B(0,4\eps)$ and $\partial B(0,1/2)$ of the pattern $(01\cdots 01)$ clockwise. See Fig.~\ref{fig::boundary_arms01}. 

\begin{figure}
\begin{center}
\includegraphics[width=0.6\textwidth]{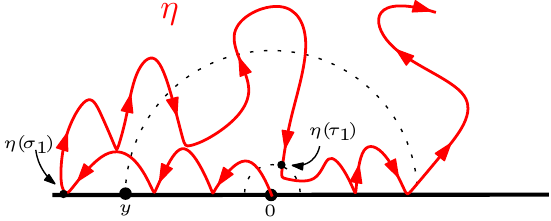}
\end{center}
\includegraphics[width=0.32\textwidth]{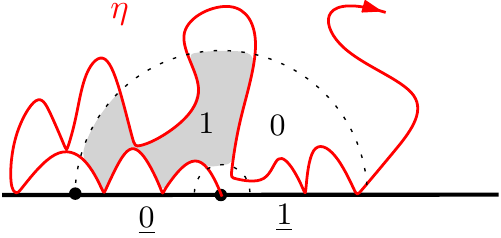}
\includegraphics[width=0.32\textwidth]{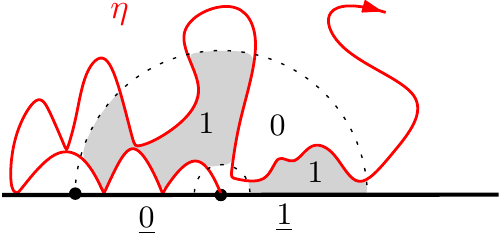}
\includegraphics[width=0.32\textwidth]{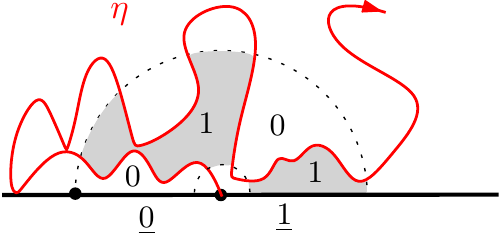}
\caption{\label{fig::boundary_arms01}
In the panel on top, the curve $\eta$ is an $\SLE_{\kappa}$ in $\HH$ from $0$ to $\infty$. Let $\sigma_1$ be the first time that $\eta$ hits $(-\infty, y)$. Let $\tau_1$ be the first time after $\sigma_1$ that $\eta$ hits the connected component of $\partial B(x,\eps)\setminus\eta[0,\sigma_1]$ containing $x+\eps$. 
In the panel on bottom left, this figure indicates the event 
$\{\tau_1<T_x\}$ which corresponds to the case of two arms 
$\alpha_2^+$ of the pattern $(10)$. 
In the panel on bottom middle, this figure indicates the event 
$\{\sigma_2<T_x\}$ which corresponds to the case of three arms 
$\beta_3^+$ of the pattern $(101)$.
In the panel on bottom right, this figure indicates the event 
$\{\sigma_1=T_u, \sigma_2<T_x\}$ which corresponds to the case of four arms $\gamma_4^+$ of the pattern $(0101)$. }
\end{figure}

It is proved in \cite[Theorems 1.1, 1.2]{WuZhanSLEBoundaryArmExponents} that, for any $y\le 0<\eps\le x$ and $j\ge 1$, we have
\begin{equation}\label{eqn::sle_boundary_alphaeven_proba}
\PP[\LH^{\alpha}_{2j}(\eps, x, y)]\asymp \left(\frac{x}{x-y}\right)^{\alpha_{2j}^+}\left(\frac{\eps}{x}\right)^{\alpha_{2j-1}^+},
\end{equation}
\begin{equation}\label{eqn::sle_boundary_betaodd_proba}
\PP\left[\LH^{\beta}_{2j-1}(\eps, x, y)\right]\asymp \left(\frac{x}{x-y}\right)^{\beta_{2j-1}^+}\left(\frac{\eps}{x}\right)^{\beta_{2j-2}^+},
\end{equation}
where the constants in $\asymp$ are uniform over $x, y, \eps$. In particular, fix some $\delta>0$, we have, for $\delta\eps\le x\le \eps/\delta$, and $-1/\delta\le y\le -\delta$,
\begin{equation}\label{eqn::sle_boundary_meaning2}
\PP[\LH^{\alpha}_{2j}(\eps, x, y)]\asymp\eps^{\alpha_{2j}^+},\quad \PP\left[\LH^{\beta}_{2j-1}(\eps, x, y)\right]\asymp\eps^{\beta_{2j-1}^+},
\end{equation}
where the constants in $\asymp$ are uniform over $\eps$. 
The readers may take~\eqref{eqn::sle_boundary_meaning2} as the definition of~\eqref{eqn::sle_boundary_alphaeven} and~\eqref{eqn::sle_boundary_betaodd}.
In Section~\ref{subsec::sle_boundary_gammaeven_proba}, we will prove the following estimate on $\LH^{\gamma}_{2j}$. 
\begin{proposition}\label{prop::sle_boundary_gammaeven_proba}
Fix some $\delta>0$ small, let $S$ be the first time that $\eta$ exits the unit disc and define
\[\LF=\{\im\eta(S)\ge \delta, S<T_u, S<T_x\}.\]
Then, for $-\eps/\delta\le u\le -\eps$, $\eps\le x\le \eps/\delta$, and $-1/\delta\le y\le -2$,
\begin{equation}\label{eqn::sle_boundary_gammaeven_proba}
\PP\left[\LH^{\gamma}_{2j}(\eps, x, y, u)\cap\LF\right]\asymp \eps^{\gamma_{2j}^+},
\end{equation} 
where the constants in $\asymp$ are uniform over $\eps$. The readers may take~\eqref{eqn::sle_boundary_gammaeven_proba} as the definition of~\eqref{eqn::sle_boundary_gammaeven}.
Moreover, for $-\eps/\delta\le u\le -\eps$, $\eps\le x\le \eps/\delta$, and $-1/\delta\le y\le -2$, we also have
\begin{equation}\label{eqn::sle_boundary_gammaeven_upper}
\PP\left[\LH^{\gamma}_{2j}(\eps, x, y, u)\right]=\eps^{\gamma_{2j}^++o(1)}.
\end{equation}
\end{proposition}

Note that \eqref{eqn::sle_boundary_gammaodd_proba} and \eqref{eqn::sle_boundary_gammaeven_proba} are weaker than \eqref{eqn::sle_boundary_alphaodd_proba}, \eqref{eqn::sle_boundary_betaeven_proba}, \eqref{eqn::sle_boundary_alphaeven_proba} and~\eqref{eqn::sle_boundary_betaodd_proba}, 
as we will only derive the upper bound of the probabilities of $\LH^{\gamma}_{2j+1}(\eps, x, y, u)$ and $\LH^{\gamma}_{2j}(\eps, x, y, u)$ restricted to the event $\LF$. 
But they are sufficient to derive the arm exponents for the critical FK-Ising model, see \cite[Section 5, Proof of Theorem 1.1]{WuAlternatingArmIsing}. In fact, assuming the convergence of interface and using RSW of the random-cluster model, we can argue that~\eqref{eqn::sle_boundary_gammaodd_proba} and~\eqref{eqn::sle_boundary_gammaeven_proba} hold without the intersection with $\LF$ \cite[Section 5]{WuAlternatingArmIsing}. However, we do not know an easy way to prove it only using $\SLE$ techniques (it is possible to argue it using Imaginary Geometry \cite{MillerSheffieldIG1}, but we prefer not to do it in this way). As RSW is known for random-cluster model \cite{DuminilSidoraviciusTassionContinuityPhaseTransition}, we decide not to spend energy in improving~\eqref{eqn::sle_boundary_gammaodd_proba} and~\eqref{eqn::sle_boundary_gammaeven_proba} to the case when there is no intersection with $\LF$. Nevertheless, we will give the proof of~\eqref{eqn::sle_boundary_gammaeven_upper}, which does not require much extra effort, as it will be needed in Section~\ref{sec::sle_interior_arm}.

The rest of this section is organized as follows. 
We prove~\eqref{eqn::sle_boundary_gammaodd_proba} in Section~\ref{subsec::sle_boundary_gammaodd_proba} where we need \eqref{eqn::sle_boundary_betaodd_proba} and an estimate on the expectation of derivatives Lemma~\ref{lem::gammaodd_aux}. 
We prove~\eqref{eqn::sle_boundary_gammaeven_proba} and~\eqref{eqn::sle_boundary_gammaeven_upper} in Section~\ref{subsec::sle_boundary_gammaeven_proba} where we need \eqref{eqn::sle_boundary_betaeven_proba} and an estimate on the expectation of derivatives Lemma~\ref{lem::gammaeven_aux}. 

\subsection{Proof of Proposition~\ref{prop::sle_boundary_gammaodd_proba}}
\label{subsec::sle_boundary_gammaodd_proba}
\begin{lemma}\label{lem::gammaodd_aux}
Fix $\kappa>4$.  
For $\lambda\ge 0$, define 
\[u_1(\lambda)=\frac{\kappa^2-6\kappa+16}{4\kappa}+\frac{\kappa-2}{2\kappa}\sqrt{4\kappa\lambda+(\kappa/2-4)^2}.\]
Let $\eta$ be an $\SLE_{\kappa}$ in $\HH$ from 0 to $\infty$. For $\eps>0$, let $\tau$ be the first time that $\eta$ hits $B(1,\eps)$ and let $T$ be the swallowing time of $1/2$. 
Fix some $\delta>0$ small, define 
\[\LG=\{\tau<T, \im{\eta(\tau)}\ge\delta\eps\},\]
\[\LF=\LG\cap\{\eta[0,\tau]\subset B(0, 1/\delta), \dist(\eta[0,\tau], [1-\eps, 1+3\eps])\ge\delta\eps\}.\] 
Then we have 
\[\eps^{u_1(\lambda)}\asymp\E\left[g'_{\tau}(1)^{\lambda}\one_{\LF}\right]\le\E\left[g'_{\tau}(1)^{\lambda}\one_{\LG}\right]\asymp\eps^{u_1(\lambda)} ,\]
where the constants in $\asymp$ are uniform over $\eps$. 
\end{lemma}
\begin{proof}
Set $\nu=\kappa-4,\rho=-\kappa/2-\sqrt{4\kappa\lambda+(\kappa/2-4)^2}$ and
\[M_t=(g_t(1/2)-W_t)^{\nu/\kappa}g_t'(1)^{\rho(\rho+4-\kappa)/(4\kappa)}(g_t(1)-W_t)^{\rho/\kappa}(g_t(1)-g_t(1/2))^{\rho\nu/(2\kappa)}.\]
By Lemma~\ref{lem::sle_mart}, the process $M$ is a local martingale and the law of $\eta$ weighted by $M$ becomes the law of $\SLE_{\kappa}(\nu, \rho)$ with force points $(1/2, 1)$. Denote by $O_t^R$ the rightmost point of $g_t(\eta[0,t])\cap\R$. 
On the event $\{\im{\eta(\tau)}\ge\delta\eps\}$, we know that 
\[\left(g_{\tau}(1/2)-W_{\tau}\right)\asymp \left(g_{\tau}(1)-W_{\tau}\right)\asymp \left(g_{\tau}(1)-g_{\tau}(1/2)\right)\asymp 
\left(g_{\tau}(1)-O_{\tau}^R\right),\]
where the constants in $\asymp$ depend only on $\delta$. By the Koebe 1/4 theorem, we have 
$\left(g_{\tau}(1)-O_{\tau}^R\right)\asymp g'_{\tau}(1)\eps$.
Therefore, by the choice of $\nu$ and $\rho$, we have, 
\[M_{\tau}\asymp g'_{\tau}(1)^{\lambda}\eps^{-u_1(\lambda)},\quad \text{on }\LG.\]
Thus 
\[\eps^{u_1(\lambda)}\PP^*[\LF^*]\asymp\E\left[g_{\tau}'(1)^{\lambda}\one_{\LF}\right]\le\E\left[g_{\tau}'(1)^{\lambda}\one_{\LG}\right] \asymp \eps^{u_1(\lambda)}\PP^*[\LG^*],\]
where $\eta^*$ is an $\SLE_{\kappa}(\nu, \rho)$ with force points $(1/2, 1)$, $\PP^*$ denotes the law of $\eta^*$ and $\LF^*, \LG^*$ are defined for $\eta^*$ accordingly. To show the conclusion, it is sufficient to show $\PP^*[\LF^*]\asymp 1$. Define $\varphi(z)=\eps z/(1-z)$. Then $\varphi$ is the Mobius transformation of the upper half plane that sends the triple $(1/2, 1, \infty)$ to $(\eps, \infty, -\eps)$. Let $\tilde{\eta}=\varphi(\eta^*)$, then $\tilde{\eta}$ is an $\SLE_{\kappa}(\kappa-6-\rho; \nu)$ with force points $(-\eps;\eps)$. Let $\tilde{S}$ be the first time that $\tilde{\eta}$ exits the unit disc and define 
$\tilde{\LF}=\{\tilde{\eta}[0,\tilde{S}]\subset U(\delta)\}$
where $U(\delta)$ is the $\delta$-neighborhood of the segment $[0,i]$. 
Note that $\nu\ge \kappa/2-2$ and $\kappa-6-\rho\ge \kappa/2-2$,  by Lemma \ref{lem::sle_goodbehavior}, we have
$\PP^*[\LF^*]\ge \tilde{\PP}[\tilde{\LF}]\ge p_0(\delta)$.
This completes the proof. 
\end{proof}
\begin{proof}[Proof of~\eqref{eqn::sle_boundary_gammaodd_proba}--Upper Bound]
Fix $\kappa\in (4,8)$ and let $\eta$ be an $\SLE_{\kappa}$ in $\HH$ from 0 to $\infty$. Let $\tau$ be the first time that $\eta$ hits $B(1,\eps)$, let $T$ be the swallowing time of $u$. Recall that 
\[\LF=\{\tau<T, \eta[0,\tau]\subset B(0,1/\delta), \dist(\eta[0,\tau], [1-\eps, 1+3\eps])\ge\delta\eps\}.\]
Given $\eta[0,\tau]$, let $f=g_{\tau}-W_{\tau}$. We know that the image of $\eta[\tau,\infty)$ under $f$, denoted by $\tilde{\eta}$, has the law of $\SLE_{\kappa}$. Define $\tilde{\LH}^{\beta}_{2j-1}$ for $\tilde{\eta}$. 
Given $\eta[0,\tau]$ and on the event $\LF$, we have the following observations.

\begin{itemize}
\item Consider the image of $\partial B(1,\eps)$ under $f$. By Lemma \ref{lem::extremallength_argument}, we know that $f(B(1,\eps))$ is contained in the ball with center $f(1+\eps)$ and radius $3(f(1+3\eps)-f(1+\eps))$. On the event $\{\dist(\eta[0,\tau], [1-\eps, 1+3\eps])\ge\delta\eps\}$, by the Koebe distortion theorem \cite[Chapter I Theorem 1.3]{Pommerenke}, we know that there exists a universal constant $C$ depending only on $\delta$ such that $f(1+\eps)-f(1)\le f(1+3\eps)-f(1)\le Cf'(1)\eps$. This implies that, on $\LF$, the image $f(B(1,\eps))$ is contained in the ball with center $f(1)$ and radius $Cf'(1)\eps$ for another constant $C$ depending only on $\delta$.
\item Consider $f(y)$. On the event $\{\eta[0,\tau]\subset B(0,1/\delta)\}$, we know that $|f(y)|$ is  bounded both sides by universal constants depending only on $\delta$. 
\end{itemize}
Combining these two observations with \eqref{eqn::sle_boundary_betaodd_proba}, we have 
\[\PP\left[\LH^{\gamma}_{2j+1}(\eps, x, y, u)\cond \eta[0,\tau], \LF\right]\le \PP\left[\tilde{\LH}^{\beta}_{2j-1}(Cf'(1)\eps, f(1), f(y))\right]\lesssim \left(g'_{\tau}(1)\eps\right)^{\beta^+_{2j-1}}.\]
Therefore, by Lemma \ref{lem::gammaodd_aux}, we have
\[\PP\left[\LH^{\gamma}_{2j+1}(\eps, x, y, u)\right]\lesssim \E\left[\left(g'_{\tau}(1)\eps\right)^{\beta^+_{2j-1}}\one_{\LF}\right]\asymp \eps^{u_1(\beta^+_{2j-1})+\beta^+_{2j-1}}.\]
Note that 
$u_1(\beta^+_{2j-1})+\beta^+_{2j-1}=\gamma_{2j+1}^+$.
This completes the proof. 
\end{proof}

\begin{proof}[Proof of \eqref{eqn::sle_boundary_gammaodd_proba}--Lower Bound]
Assume the same notations as in the proof of the upper bound. Given $\eta[0,\tau]$ and on the event $\LF$, we have the following observations.
\begin{itemize}
\item Consider the image of $\partial B(1,\eps)$ under $f$. By the Koebe 1/4 theorem, we know that $f(B(1,\eps))$ contains the ball with center $f(1)$ and radius $f'(1)\eps/4$. On the event $\{\im{\eta(\tau)}\ge\delta\eps\}$, we know that $f(1)\asymp f'(1)\eps$. 
\item Consider $f(y)$. On the event $\{\eta[0,\tau]\subset B(0,1/\delta)\}$, we know that $|f(y)|$ is bounded both sides by universal constants depending only on $\delta$.  
\end{itemize}
Combining these two facts with \eqref{eqn::sle_boundary_betaodd_proba}, we have 
\[\PP\left[\LH^{\gamma}_{2j+1}(\eps, x, y, u)\cond \eta[0,\tau], \LF\right]\ge \PP\left[\tilde{\LH}^{\beta}_{2j-1}(f'(1)\eps/4, f(1), f(y))\right]\asymp (g'_{\tau}(1)\eps)^{\beta^+_{2j-1}}.\]
Therefore, by Lemma \ref{lem::gammaodd_aux}, we have
\[\PP\left[\LH^{\gamma}_{2j+1}(\eps, x, y, u)\cap \LF\right]\gtrsim \E\left[(g'_{\tau}(1)\eps)^{\beta^+_{2j-1}}\one_{\LF}\right]\asymp \eps^{\gamma_{2j+1}^+}.\]
\end{proof}

\subsection{Proof of Proposition~\ref{prop::sle_boundary_gammaeven_proba}}
\label{subsec::sle_boundary_gammaeven_proba}

\begin{lemma}\label{lem::gammaeven_aux}
Fix $\kappa>4$.  
For $\lambda\ge 0$, define 
\[u_2(\lambda)=\frac{(\kappa-4)(\kappa+2)}{4\kappa}+\frac{(\kappa-2)}{2\kappa}\sqrt{4\kappa\lambda+(\kappa/2-2)^2}.\]
Let $\eta$ be an $\SLE_{\kappa}$ in $\HH$ from 0 to $\infty$. For $\eps>0$ small, let $T^L$ be the swallowing time of $-\eps$ and let $T^R$ be the swallowing time of $+\eps$. For $C\ge 2\eps$, let $\xi$ be the first time that $\eta$ exits $B(0,C)$.
Fix some $\delta>0$ small, define
\[\LF=\{\xi<T^L, \xi<T^R, \im{\eta(\xi)}\ge \delta C\},\quad \LG=\LF\cap\{\dist(\eta[0,\xi], \eps)\ge \eps/4\}.\]
Then we have
\[C^{-A}\eps^{u_2(\lambda)}\lesssim \E[g_{\xi}'(\eps)^{\lambda}\one_{\LG}]\le  \E[g_{\xi}'(\eps)^{\lambda}\one_{\LF}]\lesssim (\delta C)^{-A} \eps ^{u_2(\lambda)},\]
where $A$ is some constant depending on $\kappa$ and $\lambda$, the constants in $\lesssim$ are uniform over $\eps, C, \delta$. 
\end{lemma}
\begin{proof}
Set $\rho^L=\kappa-4, \rho^R=\kappa/2-2+\sqrt{4\kappa\lambda+(\kappa/2-2)^2}$ and
\[M_t=(W_t-g_t(-\eps))^{\rho^L/\kappa}g_t'(\eps)^{\rho^R(\rho^R+4-\kappa)/(4\kappa)}(g_t(\eps)-W_t)^{\rho^R/\kappa}(g_t(\eps)-g_t(-\eps))^{\rho^L\rho^R/(2\kappa)}.\]
By Lemma \ref{lem::sle_mart}, we know that $M$ is local martingale and the law of $\eta$ weighted by $M$ becomes the law of $\SLE_{\kappa}(\rho^L;\rho^R)$ with force points $(-\eps;\eps)$. By \cite[Lemma 3.4]{MillerWuSLEIntersection}, we have 
\[\delta C\lesssim \left(g_{\xi}(\eps)-W_{\xi}\right), \left(W_{\xi}-g_{\xi}(-\eps)\right)\le \left(g_{\xi}(\eps)-g_{\xi}(-\eps)\right)\le 4C,\quad \text{on }\LF.\]
By the choice of $\rho^R$, we have $\rho^R(\rho^R+4-\kappa)=4\kappa\lambda$. Thus,
\[(\delta C)^A g_{\xi}'(\eps)^{\lambda}\lesssim M_{\xi}\lesssim C^A g_{\xi}'(\eps)^{\lambda}, \quad\text{on }\LF,\]
where 
$A=\rho^L/\kappa+\rho^R/\kappa+\rho^L\rho^R/(2\kappa)$.
Therefore
\[C^{-A} M_0\PP^*[\LG^*]\lesssim \E[g_{\xi}'(\eps)^{\lambda}\one_{\LG}]\le \E[g_{\xi}'(\eps)^{\lambda}\one_{\LF}]\lesssim (\delta C)^{-A} M_0,\]
where $\eta^*$ is an $\SLE_{\kappa}(\rho^L;\rho^R)$ with force points $(-\eps;\eps)$, $\PP^*$ denotes its law and  $ \LG^*$ is defined for $\eta^*$ accordingly. 
Note that $M_0=\eps ^{u_2(\lambda)}$. To show the conclusion, it is sufficient to show 
$\PP^*[\LG^*]\asymp 1$ which is guaranteed by Lemma \ref{lem::sle_goodbehavior2} since $\rho^L\ge \kappa/2-2$ and $\rho^R\ge\kappa/2-2$. 
\end{proof}
\begin{remark}
Taking $\lambda=0$ in Lemma \ref{lem::gammaeven_aux}, we see~\eqref{eqn::sle_boundary_gammaeven_proba} holds for $\gamma_2^+=(\kappa-4)/2$. 
\end{remark}

\begin{proof}[Proof of \eqref{eqn::sle_boundary_gammaeven_proba}--Lower Bound]
Fix $\kappa\in (4,8)$ and let $\eta$ be an $\SLE_{\kappa}$ in $\HH$ from 0 to $\infty$. Let $S$ be the first time that $\eta$ exits the unit disc. Fix $x=\eps$ and $u=-\eps$ and let $T^L$ be the swallowing time of $-\eps$ and $T^R$ be the swallowing time of $\eps$. Let $\sigma$ be the first time that $\eta$ hits $(-\infty, y)$. Recall that 
$\LF=\{\im{\eta(S)}\ge \delta, S<T^L, S<T^R\}$.
Set $f_S=g_S-W_S$ and $f_{\sigma}=g_{\sigma}-W_{\sigma}$. Given $\eta[0,\sigma]$, the image of $\eta[\sigma,\infty)$ under $f_{\sigma}$, denoted by $\tilde{\eta}$, has the law of $\SLE_{\kappa}$, and we define $\tilde{\LH}^{\beta}_{2j-2}$ for $\tilde{\eta}$.  
We will control the behavior of $\eta[0,S]$ and $\eta[S, \sigma]$ separately.
\begin{itemize}
\item Consider $\eta[0,S]$ and define $\LG=\LF\cap\{\dist(\eta[0,S], x)\ge \eps/4\}$. Given $\eta[0,S]$ and on the event $\LG$, consider the image of $B(x, \eps)$ under $f_S$. On the event $\LG$, by the Koebe 1/4 theorem, we know that $f_S(B(x, \eps))$ contains the ball with center $f_S(x)$ and radius $f_S'(x)\eps/16$. Note that on the event $\LG$, we know that $|f_S(x)|$ is bounded both sides by universal constants depending only on $\delta$.

\item Given $\eta[0,S]$ and on $\LG$, consider $\eta[S, \sigma]$. From the above item, we know that $f_S(B(x, \eps))$ contains the ball with center $w:=f_S(x)$ and radius $r:=f_S'(x)\eps/16$. Define $\LE$ to be the event that $\sigma=T^L$ and that the distance between $f_S(\eta[S, \sigma])$ and $B(w, r)$ is at least $w/4$. Clearly, the probability of $\LE$ is bounded from below by a universal positive constant depending only on $\delta$ as long as $|f_S(y)|$ is bounded from above by a constant depending only on $\delta$, $r\le w/16$ and $w$ is bounded from below by a universal constant depending only on $\delta$. On the event $\LE$, note that $h:=f_{\sigma}\circ f_S^{-1}$ is the conformal map from $\HH\setminus f_S(\eta[S, \sigma])$ onto $\HH$, and the image of $B(w,r)$ under $h$ contains the ball with center $h(w)=f_{\sigma}(x)$ and radius $rh'(u)/4$. Note that on $\LE$, $f_{\sigma}(x)$ is bounded both sides by universal constants depending only on $\delta$; and, by the Koebe 1/4 theorem, the derivative $h'(u)$ is bounded from below by $f_{\sigma}(x)/(4u)$ which is therefore bounded from below by universal constant depending only on $\delta$. To summarize, given $\eta[0,\sigma]$ and on the event $\LG\cap\LE$, we know that $f_{\sigma}(B(x, \eps))$ contains a ball with center $f_{\sigma}(x)$ and radius $c_{\delta} f_S'(x)\eps $ where $f_{\sigma}(x)$ is bounded both sides by universal constants depending only on $\delta$ and $c_{\delta}>0$ depends only on $\delta$. 
\end{itemize} 
Combining these two facts with \eqref{eqn::sle_boundary_betaeven_proba}, we have that 
\[\PP\left[\LH^{\gamma}_{2j}(\eps,x,y,u)\cond \eta[0,\sigma], \LG\cap\LE\right]\gtrsim \PP\left[\tilde{\LH}^{\beta}_{2j-2}(c_{\delta} f'_S(x)\eps, f_{\sigma}(x), f_{\sigma}(y))\right]\gtrsim \left(g_S'(x)\eps\right)^{\beta^+_{2j-2}}.\]
Since the probability for $\LE$ is bounded from below by positive constant depending only on $\delta$, we have 
\[\PP\left[\LH^{\gamma}_{2j}(\eps,x,y,u)\cond \eta[0,S], \LG\right]\gtrsim \left(g_S'(x)\eps\right)^{\beta^+_{2j-2}}.\]
Therefore, by Lemma \ref{lem::gammaeven_aux}, we have 
\[\PP\left[\LH^{\gamma}_{2j}(\eps,x,y,u)\cap \LG\right]\gtrsim \E\left[\left(g_S'(x)\eps\right)^{\beta^+_{2j-2}}\one_{\LG}\right]\asymp \eps^{\gamma^+_{2j}}.\]
This completes the proof. 
\end{proof}

\begin{lemma}\label{lem::sle_boundary_gammaeven_upper_aux1}
Fix $\kappa>4$, let $\eta$ be an $\SLE_{\kappa}$ in $\HH$ from 0 to $\infty$. For $y\le -2, u\in [-\eps/\delta, -\eps], x\in [\eps, \eps/\delta]$, let $T^L$ be the swallowing time of $u$ and let $T^R$ be the swallowing time of $x$. For $C\in [2\eps, 1]$, let $\xi$ be the first time that $\eta$ exits $B(0,C)$. 
Fix some $\delta>0$ small, define
\[\LF=\{\xi<T^L, \xi<T^R, \im{\eta(\xi)}\ge \delta C\}.\]
Then we have, for $j\ge 2$,
\[\PP\left[\LH^{\gamma}_{2j}(\eps, x, y, u)\cap \LF\right]\lesssim C^{-A}\delta^{-B} \eps^{\gamma_{2j}^+}.\]
where $A, B$ are some constants depending on $\kappa$ and $j$, the constants in $\lesssim$ are uniform over $\eps, C, \delta$. Note that this lemma gives the upper bound in \eqref{eqn::sle_boundary_gammaeven_proba}. 
\end{lemma}
\begin{proof}
Given $\eta[0,\xi]$, let $f=g_{\xi}-W_{\xi}$. We know that the image of $\eta[\xi, \infty)$ under $f$, denoted by $\tilde{\eta}$, has the same law as  $\SLE_{\kappa}$. Define $\tilde{\LH}^{\beta}_{2j-2}$ for $\tilde{\eta}$. We have the following observations.
\begin{itemize}
\item Consider the image of $\partial B(x, \eps)$ under $f$. By Lemma \ref{lem::extremallength_argument}, we know that the image of $\partial B(x, \eps)$ under $f$ is contained in the ball with center $f(x+3\eps)$ and radius $8\eps f'(x+3\eps)$. On the event $\LF$, we have $\delta C\lesssim f(x+3\eps)\le 2C$.
\item Consider $f(y)$. The quantity $|f(y)|$ is bounded from below by universal constant as long as $y\le -2$.
\end{itemize}
Combining these two facts with \eqref{eqn::sle_boundary_betaeven_proba}, we have 
\begin{align*}
\PP\left[\LH^{\gamma}_{2j}(\eps, x, y, u)\cond \eta[0,S], \LF\right]&\le\PP\left[\tilde{\LH}^{\beta}_{2j-2}(8\eps f'(x+3\eps), f(x+3\eps), f(y))\right]\\
&\lesssim (\delta C)^{-A}\left(g'_S(x+3\eps)\eps\right)^{\beta^+_{2j-2}},
\end{align*}
where $A$ is some constant depending on $\kappa$ and $j$.  
Therefore, by Lemma \ref{lem::gammaeven_aux}, we have
\[\PP\left[\LH^{\gamma}_{2j}(\eps, x, y, u)\cap\LF\right]\lesssim (\delta C)^{-A}\E[\left(g'_S(x+3\eps)\eps\right)^{\beta^+_{2j-2}}\one_{\LF}]\lesssim C^{-A}\delta^{-B}\eps^{\gamma^+_{2j}},\]
where 
$A, B$ are some constants depending only on $\kappa, j$. 
This completes the proof. 
\end{proof}
\begin{lemma}\label{lem::sle_boundary_gammaeven_upper_aux2}
Fix $\kappa\in (0,8)$ and let $\eta$ be an $\SLE_{\kappa}$ in $\HH$ from 0 to $\infty$. Fix $n\ge 1$ such that $2^{-n}\ge 2\eps$. For $1\le m\le n$, let $\xi_m$ be the first time that $\eta$ exits $B(0, 2^{m-n+1})$. Note that $\xi_1, ..., \xi_n$ is an increasing sequence of stopping times and $\xi_1$ is the first time that $\eta$ exits $B(0, 2^{-n})$ and $\xi_n$ is the first time that $\eta$ exits $B(0,1/2)$. For $1\le m\le n$, define 
\[\LF_m=\{\im{\eta(\xi_m)}\le \delta 2^{m-n+1}\}.\]
There exists a function $p: (0,1)\to [0,1]$ with $p(\delta)\downarrow 0$ as $\delta\downarrow 0$ such that 
$\PP[\cap_1^n\LF_m]\le p(\delta)^n$.
\end{lemma}
\begin{proof}
For $1\le m\le n$, given $\eta[0,\xi_m]$, let $f_m=g_{\xi_m}-W_{\xi_m}$. Denote $2^{m-n+1}$ by $r$. The event $\LF_{m+2}$ is that $\eta$ exits $B(0,4r)$ through $B(-4r, \delta 4r)\cup B(4r, \delta 4r)$. Let $\tilde{\eta}$ be the image of $\eta[\xi_m,\infty)$ under $f_m$. Then $\LF_{m+2}$ implies that $\tilde{\eta}$ hits $f_m(B(-4r, \delta 4r))\cup f_m(B(4r, \delta 4r))$. Consider $f_m(B(4r, \delta 4r))$. By Lemma \ref{lem::image_insideball}, we know that $f_m(B(4r, \delta 4r))$ is contained in the ball with center $f_m(4r)$ and radius $16\delta r f_m'(4r)$. 
By \cite[Corollary 3.44]{LawlerConformallyInvariantProcesses}, we have that 
$4r\le f_m(4r)\le 8r$, and $f_m'(4r)\asymp 1$.
Thus, by \cite{AlbertsKozdronIntersectionProbaSLEBoundary}, we have 
$\PP\left[\LF_{m+2}\cond \eta[0,\xi_m]\right]\le C\delta^{8/\kappa-1}$.
Iterating this relation, we have 
\[\PP\left[\cap_1^n\LF_m\right]\le \left(C\delta^{8/\kappa-1}\right)^{n/2}.\] This implies the conclusion. 
\end{proof}

\begin{proof}[Proof of \eqref{eqn::sle_boundary_gammaeven_upper}--Upper Bound]
Assume the same notation as in Lemma~\ref{lem::sle_boundary_gammaeven_upper_aux2}. For $1\le m\le n$, by Lemma~\ref{lem::sle_boundary_gammaeven_upper_aux1}, we have that 
\[\PP\left[\LH^{\gamma}_{2j}(\eps, x, y, u)\cap \LF_m^c\right]\lesssim 2^{nA}\delta^{-B}\eps^{\gamma^+_{2j}},\]
as long as $y\le -2$, where $A, B$ are some constants depending on $\kappa, j$. 
Combining with Lemma \ref{lem::sle_boundary_gammaeven_upper_aux2}, we have, for any $n$ and $\delta>0$ small,
\[\PP\left[\LH^{\gamma}_{2j}(\eps, x, y, u)\right]\lesssim n2^{nA}\delta^{-B}\eps^{\gamma^+_{2j}}+p(\delta)^n,\]
where $p(\delta)\downarrow 0$ as $\delta\downarrow 0$. 
This implies the conclusion. 
\end{proof}

%%%%%
\section{Proof of Theorem \ref{thm::sle_interiorarm}}
\label{sec::sle_interior_arm}
%%%%%
Fix $\kappa\in (4,8)$ and let $\eta$ be an $\SLE_{\kappa}$ in $\HH$ from 0 to $\infty$. Fix $z\in \HH$ with $|z|=1$ and suppose $y\le 0$. Let $T_z$ be the first time that $\eta$ swallows $z$. We are interested in the crossings of $\eta$ between the ball $B(z,\eps)$ and the interval $(-\infty, y)$. We write c.c. for ``connected component". 

Set $\tau_0=\sigma_0=0$. Let $\tau_1$ be the first time that $\eta$ hits $B(z,\eps)$ and let $\sigma_1$ be the first time after $\tau_1$ that $\eta$ hits $(\infty, y)$. Given $\eta[0,\sigma_1]$ and suppose $\sigma_1<T_z$, we know that $B(z,\eps)\setminus \eta[0,\sigma_1]$ has one c.c. that contains $z$, denoted by $C_z$. The boundary $\partial C_z$ consists of pieces of $\eta[0,\sigma_1]$ and pieces of $\partial B(z,\eps)$. Consider $\partial C_z\cap \partial B(z,\eps)$, there may be several c.c.s, but there is only one which can be connected to $\infty$ in $\HH\setminus (\eta[0,\sigma_1]\cup B(z,\eps))$. We denote this c.c. by $C_z^b$ and orient it clockwise and denote the ending point of $C_z^b$ by $X_z^b$. See Figure \ref{fig::sle_interior_def}. Let $\tau_2$ be the first time after $\sigma_1$ that $\eta$ hits $C_z^b$, and let $\sigma_2$ be the first time after $\tau_2$ that $\eta$ hits $(-\infty, y)$. For $j\ge 2$, let $\tau_j$ be the first time after $\sigma_{j-1}$ such that $\eta$ hits the c.c. of $C_z^b\setminus \eta[0,\sigma_{j-1}]$ containing $X_z^b$ and let $\sigma_j$ be the first time after $\tau_j$ that $\eta$ hits $(-\infty, y)$.
Define 
\[\LE^{\alpha}_{2j}(\eps, z, y)=\{\tau_j<T_z\}, \quad \LE^{\beta}_{2j+1}(\eps, z, y)=\{\sigma_j<T_z\}. \]

\begin{figure}[ht!]
\begin{center}
\includegraphics[width=0.4\textwidth]{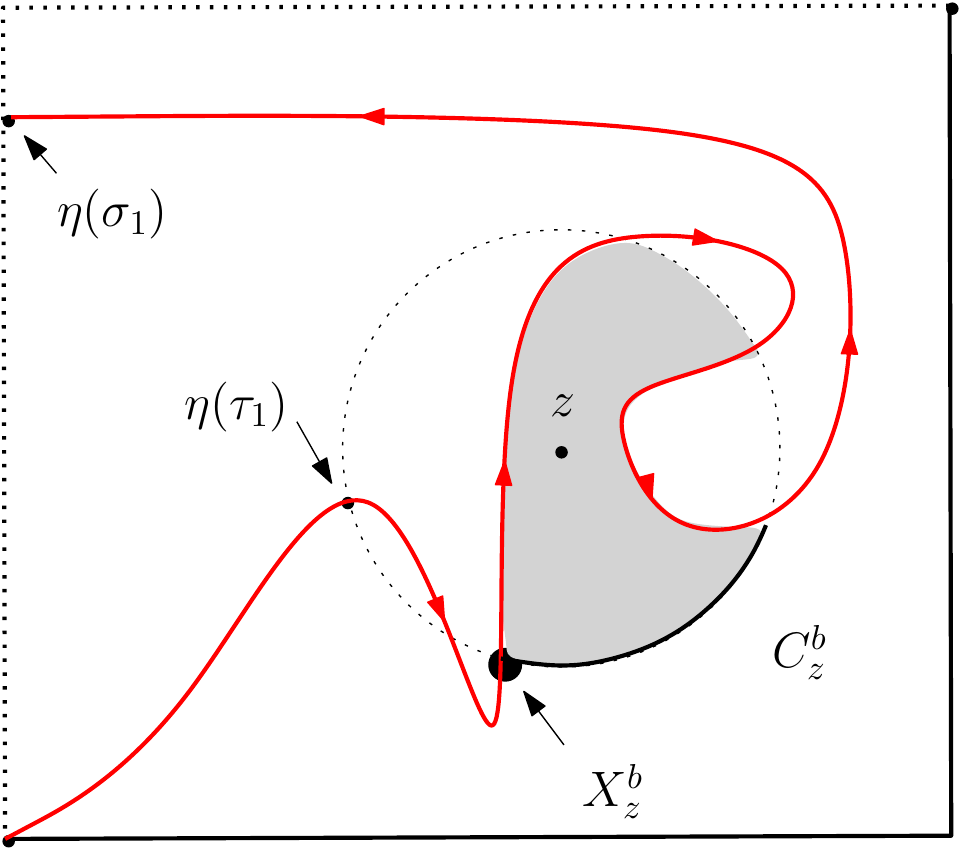}
\end{center}
\caption{\label{fig::sle_interior_def} The gray part is the connected component of $B(z,\eps)\setminus \eta[0,\sigma_1]$ that contains $z$, which is denoted by $C_z$. The bold part of $\partial C_z$ is $C_z^b$ and the point $X_z^b$ is indicated in the figure. }
\end{figure}

The definition of $\LE^{\gamma}$ is a little complicated. Given $\eta[0,\tau_1]$, let $f_{\tau_1}=g_{\tau_1}-W_{\tau_1}$ and set $u=-4|g_{\tau_1}'(z)|\eps$. Denote $f_{\tau_1}^{-1}(u)$ by $w$ and let $T_w$ be the first time that $\eta$ swallows $w$. Define 
\[\LE^{\gamma}_{2j+2}=\{\sigma_1=T_w, \sigma_j<T_z\}.\]
We will estimate the probability of $\LE^{\alpha}, \LE^{\beta}$ and $\LE^{\gamma}$, but due to technical difficulty in the proof, we need an auxiliary event. 
Define 
\[\LF=\{\eta[0,\tau_1]\subset B(0,R)\},\]
where $R$ is a constant from Lemma~\ref{lem::sle_interior_derivative} and it  depends only on $\kappa$ and $z$ which is decided in Lemma \ref{lem::sle_interior_derivative}. 

\begin{proposition}
Assume the same notations as in Theorem \ref{thm::sle_interiorarm}. For fixed $z\in \HH$ with $|z|=1$, fixed $\delta>0$ small and for $j\ge 1$, and for $-1/\delta\le y\le -2R$, we have
\begin{equation}\label{eqn::sle_interior_alpha_proba}
\PP\left[\LE^{\alpha}_{2j}(\eps, z, y)\cap \LF\right]=\eps^{\alpha_{2j}+o(1)};
\end{equation}
\begin{equation}\label{eqn::sle_interior_beta_proba}
\PP\left[\LE^{\beta}_{2j+1}(\eps, z, y)\cap\LF\right]=\eps^{\beta_{2j+1}+o(1)};
\end{equation}
\begin{equation}\label{eqn::sle_interior_gamma_proba}
\PP\left[\LE^{\gamma}_{2j+2}(\eps, z, y)\cap\LF\right]=\eps^{\gamma_{2j+2}+o(1)}. 
\end{equation}
The readers may take~\eqref{eqn::sle_interior_alpha_proba}, \eqref{eqn::sle_interior_beta_proba} and~\eqref{eqn::sle_interior_gamma_proba} as the definition of~\eqref{eqn::sle_interior_alpha}, \eqref{eqn::sle_interior_beta} and~\eqref{eqn::sle_interior_gamma}.
\end{proposition}

The rest of this section is organized as follows. We first explain the choice of the constant $R$ in Lemma~\ref{lem::sle_interior_derivative} and then give the proof for the lower bound of~\eqref{eqn::sle_interior_alpha_proba}. To derive the upper bound of~\eqref{eqn::sle_interior_alpha_proba}, we need Lemmas~\ref{lem::sle_interior_upper_aux1} and~\ref{lem::sle_interior_upper_aux2}. The proof of~\eqref{eqn::sle_interior_beta_proba} and~\eqref{eqn::sle_interior_gamma_proba} are similar. 
\begin{lemma}\label{lem::sle_interior_derivative}
\cite[Lemma 4.2]{WuAlternatingArmIsing}.
Fix $\kappa\in (0,8)$ and let $\eta$ be an $\SLE_{\kappa}$ in $\HH$ from 0 to $\infty$.
For $\lambda\ge 0$, define 
\[\rho=\kappa/2-4-\sqrt{4\kappa\lambda+(\kappa/2-4)^2},\quad v(\lambda)=\frac{1}{2}-\frac{\kappa}{16}-\frac{\lambda}{2}+\frac{1}{8}\sqrt{4\kappa\lambda+(\kappa/2-4)^2}.\]
Fix $z\in \HH$ with $|z|=1$. 
For $\eps>0$, let $\tau$ be the first time that $\eta$ hits $B(z, \eps)$. 
Define $\Theta_t=\arg(g_t(z)-W_t)$. For $\delta\in (0,1/16), R\ge 4$, define 
\[\LG=\{\tau<\infty, \Theta_{\tau}\in (\delta, \pi-\delta)\},\quad \LF=\{\eta[0,\tau]\subset B(0,R)\}.\]
There exists a constant $R$ depending only on $\kappa$ and $z$ such that the following is true:
\[\eps^{v(\lambda)}\lesssim \E\left[|g'_{\tau}(z)|^{\lambda}\one_{\LF\cap\LG}\right]\le \E\left[|g'_{\tau}(z)|^{\lambda}\one_{\LG}\right]\lesssim \eps^{v(\lambda)}\delta^{-v(\lambda)-\rho^2/(2\kappa)},\]
where the constants in $\lesssim$ are uniform over $\eps, \delta$.
\end{lemma}

Now we have decided the constant $R$ in Lemma \ref{lem::sle_interior_derivative}, and we will fix it in the following of the paper. 
\begin{proof}[Proof of \eqref{eqn::sle_interior_alpha_proba}--Lower Bound] Let $\eta$ be an $\SLE_{\kappa}$ in $\HH$ from 0 to $\infty$. Let $\tau$ be the first time that $\eta$ hits $B(z,\eps)$. 
Denote the centered conformal map $g_t-W_t$ by $f_t$ for $t\ge 0$. Recall that 
$\LF=\{\eta[0,\tau]\subset B(0,R)\}$.
Fix some $\delta>0$ and define 
$\LG=\LF\cap \{\Theta_{\tau}\in (\delta, \pi-\delta)\}$.

We run $\eta$ until the time $\tau$ and on the event $\LG$, by the Koebe 1/4 theorem, we know that $f_{\tau}(B(z,\eps))$ contains the ball with center $w:=f_{\tau}(z)$ and radius $r:=\eps |f_{\tau}'(z)|/4$ and 
$\arg(w)\in (\delta, \pi-\delta)$, and $r\le \im{w}\le 16r$. 
We wish to apply \eqref{eqn::sle_boundary_alphaeven_proba}, however this ball is centered at $w=f_{\tau}(z)$ which does not satisfy the conditions in \eqref{eqn::sle_boundary_alphaeven_proba}. We will fix this problem by running $\eta$ a little further and argue that there is positive chance that $\eta$ does the right thing. 

Let $\tilde{\eta}$ be the image of $\eta[\tau, \infty)$ under $f_{\tau}$. Let $\gamma$ be the broken line from 0 to $w$ and then to $-r$ and let $A_r$ be the $r/4$-neighborhood of $\gamma$. Let $S_1$ be the first time that $\tilde{\eta}$ exits $A_r$ and let $S_2$ be the first time that $\tilde{\eta}$ hits $(-\infty, -r)$. By \cite[Lemma 2.5]{MillerWuSLEIntersection}, we know that $\PP[S_2<S_1]$ is bounded from below by positive constant depending only on $\kappa$ and $\delta$. On the event $\{S_2<S_1\}$, it is clear that there exist constants $x_{\delta}, c_{\delta}>0$ depending only on $\delta$ such that $f_{S_2}(B(z,\eps))$ contains the ball with center $x_{\delta}r$ and radius $c_{\delta}r$. Let $\hat{\eta}$ be the image of $\eta[S_2, \infty)$ under $f_{S_2}$ and define $\hat{\LH}^{\alpha}_{2j}$ for $\hat{\eta}$. Then, by \eqref{eqn::sle_boundary_alphaeven_proba}, we have 
\[\PP\left[\LE^{\alpha}_{2j}(\eps, z, y)\cond \eta[0,S_2], \LG\cap\{S_2<S_1\}\right]\ge \PP\left[\hat{\LH}^{\alpha}_{2j-2}(c_{\delta}r, x_{\delta}r, f_{S_2}(y))\right]\gtrsim (|g_{\tau}'(z)|\eps)^{\alpha^+_{2j-2}}.\]
Since $\{S_2<S_1\}$ has a positive chance, we have
\[\PP\left[\LE^{\alpha}_{2j}(\eps, z, y)\cond \eta[0,\tau], \LG\right]\gtrsim (|g_{\tau}'(z)|\eps)^{\alpha^+_{2j-2}}.\]
Therefore, by Lemma \ref{lem::sle_interior_derivative}, we have 
\[\PP\left[\LE^{\alpha}_{2j}(\eps, z, y)\cap\LG\right]\gtrsim\E\left[(|g_{\tau}'(z)|\eps)^{\alpha^+_{2j-2}}\one_{\LG}\right]\asymp \eps^{\alpha_{2j}},\]
where the constants in $\gtrsim$ and $\asymp$ are uniform over $\eps$. This completes the proof. 
\end{proof}

%\begin{figure}[ht!]
%\begin{center}
%\includegraphics[width=\textwidth]{figures/within_tube_positivechance}
%\end{center}
%\caption{\label{fig::tube_positivechance} By Koebe 1/4 theorem, we know that $f_{\tau}(B(z,\eps))$ contains the ball $B(w, r)$ where $w=f_{\tau}(z)$ and $r=\eps |f_{\tau}'(z)|/4$ where $\arg(w)\in (\delta, \pi-\delta)$, and $r\le \im{w}\le 16r$. The event $\{S_2<S_1\}$ means that the curve $\tilde{\eta}$ hits $(-\infty, -r)$ before exiting the tube $A_r$. This event has positive chance which is bounded from below by constant depending only on $\kappa$ and $\delta$. On the event $\{S_2<S_1\}$, it is clear that $f_{S_2}(B(z,\eps))$ contains a ball with center $x_{\delta}r$ and radius $c_{\delta}r$ where $x_{\delta}, c_{\delta}$ are positive constants depending only on $\delta$.}
%\end{figure}

\begin{lemma}\label{lem::sle_interior_upper_aux1}
Fix $\kappa\in (4,8)$ and let $\eta$ be an $\SLE_{\kappa}$ in $\HH$ from 0 to $\infty$. Fix $z\in\HH$ with $|z|=1$ and let $T_z$ be the first time that $\eta$ swallows $z$. Let $\Theta_t=\arg(g_t(z)-W_t)$. 
For $C\ge 16$, let $\xi$ be the first time that $\eta$ hits $\partial B(z, C\eps)$. For $\delta\in (0,1/16)$, define 
\[\LF=\{\xi<T_z, \Theta_{\xi}\in (\delta, \pi-\delta), \eta[0,\xi]\subset B(0,R)\}.\]
Then we have, for $j\ge 1$,
\[\PP\left[\LE^{\alpha}_{2j+2}(\eps, z, y)\cap \LF\right]\lesssim C^A\delta^{-B}\eps^{\alpha_{2j+2}},\quad \text{provided }y\le -2R.\]
where $A, B$ are some constants depending on $\kappa$ and $j$, and the constant in $\lesssim$ is uniform over $\delta, C, \eps$.  
\end{lemma}
\begin{proof}
We run the curve up to time $\xi$ and let $f=g_{\xi}-W_{\xi}$. We know that the image of $\eta[\xi, \infty)$ under $f$ has the same law as $\SLE_{\kappa}$, we denote it by $\tilde{\eta}$ and define $\tilde{\LH}^{\alpha}_{2j}$ for $\tilde{\eta}$. We have the following observations.
\begin{itemize}
\item By Lemma \ref{lem::image_insideball}, we know that $f(B(z,\eps))$ is contained in the ball with center $f(z)$ and radius $r:=4\eps|f'(z)|$. Applying the Koebe 1/4 theorem to $f$, we have 
\begin{equation}\label{eqn::sle_interior_upperaux1_im}
C\eps |f'(z)|/4\le \im{f(z)}\le 4C\eps|f'(z)|.
\end{equation}
Next, we argue that $f(B(z,\eps))$ is contained in the ball with center $|f(z)|\in \R$ and radius $8Cr/\delta$. Since $f((z,\eps))$ is contained in the ball with center $f(z)$ and radius $r$, it is clear that $f(B(z,\eps))$ is contained in the ball with center $|f(z)|$ with radius $r+2|f(z)|$. By \eqref{eqn::sle_interior_upperaux1_im}, we have 
\[Cr/16\le |f(z)|\sin\Theta_{\xi}\le Cr.\]
Since $\Theta_{\xi}\in (\delta, \pi-\delta)$, we know that, for $\delta>0$ small, we have $\sin\Theta_{\xi}\ge\delta/2$.
Thus, $Cr/16\le |f(z)|\le 2Cr/\delta$. Therefore, $f(B(z,\eps))$ is contained in the ball with center $|f(z)|$ with radius $8Cr/\delta$. In summary, we know that $f(B(z,\eps))$ is contained in the ball with center $|f(z)|$ and radius $32C\eps |f'(z)|/\delta$ where 
\[C\eps |f'(z)|/4\le |f(z)|\le 8C\eps|f'(z)|/\delta.\]
\item Since $\{\eta[0,\xi]\subset B(0,R)\}$ and $y\le -2R$, it is clear that $|f(y)|$ is bounded from below by universal constant.
\end{itemize}
Combining these two facts with \eqref{eqn::sle_boundary_alphaeven_proba}, we have 
\[\PP\left[\LE^{\alpha}_{2j+2}(\eps, z, y)\cond \eta[0,\xi], \LF\right]\le \PP\left[\tilde{\LH}^{\alpha}_{2j}(32C\eps|f'(z)|/\delta, |f(z)|, f(y))\right]\lesssim \left(C\eps |g_{\xi}'(z)|/\delta\right)^{\alpha_{2j}^+},\]
where the constant in $\lesssim$ is uniform over $C, \eps, \delta$. 
Thus, by Lemma \ref{lem::sle_interior_derivative}, we have 
\[\PP\left[\LE^{\alpha}_{2j+2}(\eps, z, y)\cap \LF\right]\lesssim \left(C\eps/\delta\right)^{\alpha_{2j}^+}\E\left[|g_{\xi}'(z)|^{\alpha_{2j}^+}\one_{\LF}\right]\lesssim \delta^{-b}\left(C\eps/\delta\right)^{\alpha_{2j}^+}\eps^{v(\alpha_{2j}^+)},\]
where $b$ is some constant from Lemma \ref{lem::sle_interior_derivative}. 
Note that 
$\alpha_{2j+2}=v(\alpha_{2j}^+)+\alpha_{2j}^+$.
This completes the proof. 
\end{proof}
From Lemma \ref{lem::sle_interior_upper_aux1}, we see that in order to show the upper bound in \eqref{eqn::sle_interior_alpha_proba}, it remains to argue that $\{\Theta_{\xi}\in (\delta, \pi-\delta)\}$ happens with high probability. This is guaranteed by the following lemma.
\begin{lemma}\label{lem::sle_interior_upper_aux2}
\cite[Lemma 4.4]{WuAlternatingArmIsing}. 
Fix $\kappa\in (0,8)$ and let $\eta$ be an $\SLE_{\kappa}$ in $\HH$ from 0 to $\infty$. Fix $z\in\HH$ with $|z|=1$. Let $T_z$ be the first time that $\eta$ swallows $z$ and set $\Theta_t=\arg(g_t(z)-W_t)$.
Take $n\in\N$ such that $B(z, 16\eps 2^n)$ is contained in $\HH$. For $1\le m\le n$, let $\xi_m$ be the first time that $\eta$ hits $B(z, 16\eps 2^{n-m+1})$. Note that $\xi_1, ..., \xi_n$ is an increasing sequence of stopping times and $\xi_1$ is the first time that $\eta$ hits $B(z, 16\eps 2^n)$ and $\xi_n$ is the first time that $\eta$ hits $B(z, 32\eps)$. For $1\le m\le n$, for $\delta>0$, define 
\[\LF_m=\{\xi_m<T_z, \Theta_{\xi_m}\not\in (\delta, \pi-\delta)\}\]
There exists a function $p:(0,1)\to [0,1]$ with $p(\delta)\downarrow 0$ as $\delta\downarrow 0$ such that 
$\PP\left[\cap_{1}^n \LF_m\right]\le p(\delta)^n$.
\end{lemma}

\begin{proof}[Proof of \eqref{eqn::sle_interior_alpha_proba}--Upper Bound]
Assume the same notations as in Lemma \ref{lem::sle_interior_upper_aux2}. Recall that 
$\LF=\{\eta[0,\tau_1]\subset B(0,R)\}$.
By Lemma \ref{lem::sle_interior_upper_aux1}, we have, for $1\le m\le n$
\[\PP\left[\LE^{\alpha}_{2j+2}(\eps, z, y)\cap\LF\cap \LF_m^c\right]\lesssim 2^{nA}\delta^{-B}\eps^{\alpha_{2j}},\]
where $A, B$ are some constants depending on $\kappa$ and $j$. Combining with Lemma \ref{lem::sle_interior_upper_aux2}, we have, for any $n$ and $\delta>0$, 
\[
\PP\left[\LE^{\alpha}_{2j+2}(\eps, z, y)\cap \LF\right]
\lesssim n2^{nA}\delta^{-B}\eps^{\alpha_{2j}} +p(\delta)^n,\]
where $p(\delta)\downarrow 0$ as $\delta\downarrow 0$. This implies the conclusion. 
\end{proof}

\begin{proof}[Proof of \eqref{eqn::sle_interior_beta_proba}] 
The lower bound for \eqref{eqn::sle_interior_beta_proba} can be proved in the same way as the lower bound of \eqref{eqn::sle_interior_alpha_proba}. 
By the same proof of Lemma \ref{lem::sle_interior_upper_aux1} where we replace \eqref{eqn::sle_boundary_alphaeven_proba} by \eqref{eqn::sle_boundary_betaodd_proba}, we obtain 
\[\PP\left[\LE^{\beta}_{2j+1}(\eps, z, y)\cap\{\xi<T_z, \Theta_{\xi}\in (\delta, \pi-\delta), \eta[0,\xi]\subset B(0,R)\}\right]\lesssim C^A\delta^{-B}\eps^{\beta_{2j+1}},\]
as long as $y\le -2R$, where $A, B$ are some constants depending on $\kappa, j$. Then we can repeat the same proof of the upper bound for \eqref{eqn::sle_interior_alpha_proba} to obtain the upper bound for \eqref{eqn::sle_interior_beta_proba}. 
\end{proof}

\begin{proof}[Proof of \eqref{eqn::sle_interior_gamma_proba}]
We can repeat the same proof of the lower bound of \eqref{eqn::sle_interior_alpha_proba} to give the lower bound of \eqref{eqn::sle_interior_gamma_proba}. We only need to take care of the point $u:=-4\eps |g_{\tau}'(z)|$.  Given $\eta[0,S_2]$ and on the event $\{S_2<S_1\}$, we also have that 
$f_{S_2}\circ f^{-1}_{\tau}(u)\asymp \eps |g_{\tau}'(z)|$.
Then we can use the same argument to get the lower bound for \eqref{eqn::sle_interior_gamma_proba}.

By the same proof of Lemma \ref{lem::sle_interior_upper_aux1} where we replace \eqref{eqn::sle_boundary_alphaeven_proba} by \eqref{eqn::sle_boundary_gammaeven_upper}, we could obtain 
\[\PP\left[\LE^{\gamma}_{2j+2}(\eps, z, y)\cap\{\xi<T_z, \Theta_{\xi}\in (\delta, \pi-\delta), \eta[0,\xi]\subset B(0,R)\}\right]\le C^A\delta^{-B}\eps^{\beta_{2j+1}+o(1)},\]
as long as $y\le -2R$, where $A, B$ are some constants depending on $\kappa, j$. Then we can repeat the same proof of the upper bound for \eqref{eqn::sle_interior_alpha_proba} to obtain the upper bound for \eqref{eqn::sle_interior_gamma_proba}. 
\end{proof}

%%%%%
\begin{acknowledgements}
H. W.'s work is supported by NCCR/SwissMAP, ERC AG COMPASP, the Swiss NSF as well as the startup funding no. 042-53331001017 of Tsinghua University. 
The author acknowledges Hugo Duminil-Copin, Aran Raoufi, Stanislav Smirnov, and Vincent Tassion for helpful discussion on the critical lattice models. The author thanks Gregory Lawler, David Wilson and Dapeng Zhan for helpful discussions on SLE estimates.
The author also acknowledges two anonymous referees for the helpful comments on the earlier version of the article. 
\end{acknowledgements}

%\bibliographystyle{plain}
%\bibliography{bibliography}

\end{document}